\documentclass[11pt,letterpaper]{amsart}
\usepackage{fullpage}
\usepackage[UKenglish]{babel}
\usepackage{colonequals}
\usepackage{placeins}
\usepackage{graphicx}	
\usepackage{amsthm}
\usepackage{amsmath}
\usepackage{amssymb}
\usepackage{mathtools}
\usepackage{enumerate}
\usepackage{booktabs}
\usepackage{mathrsfs}
\usepackage{thmtools, thm-restate}
\usepackage[table]{xcolor}
\usepackage{tabularx, booktabs}
\newcolumntype{Y}{>{\centering\arraybackslash}X}
\usepackage{tikz}
\usepackage{dsfont}
\usepackage{pgfplots}
\usepackage{wrapfig}
\usepackage{caption}
\usepackage{hyperref}
\definecolor{mylinkcolor}{rgb}{0.5,0.0,0.0}
\definecolor{myurlcolor}{rgb}{0.0,0.0,0.75}
\definecolor{mydeficolor}{rgb}{0.0,0.5,0.0}
\hypersetup{colorlinks=true,urlcolor=myurlcolor,citecolor=myurlcolor,linkcolor=mylinkcolor,linktoc=page,breaklinks=true}
\usepackage{longtable}
\usepackage{tikz-cd} 
\usetikzlibrary {arrows.meta,bending,positioning,patterns,decorations}
\usepackage{soul}

\pgfplotsset{compat=1.16}

\usepackage[OT2,T1]{fontenc}
\DeclareSymbolFont{cyrletters}{OT2}{wncyr}{m}{n}
\DeclareMathSymbol{\Sha}{\mathalpha}{cyrletters}{"58}

\newcommand{\mr}[1]{MathSciNet:\,\href{https://mathscinet.ams.org/mathscinet-getitem?mr=#1}{MR#1}}
\newcommand{\arxiv}[1]{arXiv:\,\href{https://arxiv.org/abs/#1}{#1}}

\newtheorem{theorem}{Theorem}[section]
\newtheorem{corollary}[theorem]{Corollary}

\newtheorem{proposition}[theorem]{Proposition}

\newtheorem{algorithm}{Algorithm}

\theoremstyle{definition}
\newtheorem{definition}[theorem]{Definition}
\newtheorem{notation}[theorem]{Notation}
\newtheorem{example}[theorem]{Example}
\newtheorem{remark}[theorem]{Remark}

\newcommand{\Q}{\mathbb{Q}}

\newcommand{\Z}{\mathbb{Z}}
\newcommand{\cR}{\mathcal{R}}
\newcommand{\F}{\mathbb{F}}

\newcommand{\Frob}{\mathrm{Frob}}

\newcommand{\Gal}{\mathrm{Gal}}

\newcommand{\Jac}{\mathrm{Jac}\,}
\newcommand{\s}{\mathfrak{s}}

\newcommand{\OL}{\mathcal O_L}
\newcommand{\gcdk}{{\mathrm{gcd}_k}}
\newcommand{\lift}{{\mathrm{lift}}}
\newcommand{\defi}[1]{{\color{mydeficolor}\textsf{#1}}}

\newcolumntype{C}[1]{>{\centering\let\newline\\\arraybackslash\hspace{0pt}}m{#1}}

\usepackage{pbox}
\usepackage{tkz-graph}
\usetikzlibrary{decorations.pathmorphing,fit,positioning,calc,shapes}
\usepackage{graphicx}
\usepackage{cleveref,tikz-cd,pbox,wrapfig}

\definecolor{amethyst}{rgb}{0.6, 0.4, 0.8}
\definecolor{atomictangerine}{rgb}{1.0, 0.6, 0.4}
\definecolor{deeppeach}{rgb}{1.0, 0.8, 0.64}
\definecolor{eggshell}{rgb}{0.94, 0.92, 0.84}
\definecolor{lightapricot}{rgb}{0.99, 0.84, 0.69}
\definecolor{lemonchiffon}{rgb}{1.0, 0.98, 0.8}
\definecolor{roundabout}{rgb}{1.0, 0.91, 0.75}
\definecolor{atomictangerine}{rgb}{1.0, 0.6, 0.4}
\definecolor{ruby}{rgb}{0.88, 0.07, 0.37}
\definecolor{sapphire}{rgb}{0.03, 0.15, 0.4}

\def\rootsep{0.03}               % horizontal space between roots in a cluster picture
\def\clustersep{0.06}            % horizontal and vertical space between parent & child cluster
\def\cnamescale{0.4}             % cluster names font size 
\def\cdepthscale{0.4}            % cluster depths and signs font size 
\def\cltopskip{1pt}              % space above a cluster picture
\def\clbottomskip{1pt}           % space below a cluster picture

\def\rootscale{0.5}   \def\rootcolor{gray}
\def\rootscaleA{0.7}  \def\rootcolorA{sapphire}
\def\rootscaleB{0.5}  \def\rootcolorB{green}
\def\rootscaleC{0.4}  \def\rootcolorC{sapphire}
\def\rootscaleD{0.45}  \def\rootcolorD{ruby}
\tikzset{
  clA/.style = {very thick,black},
  clB/.style = {thick,purple}
}

\def\graphdslabelscale{0.6}

\def\GraphScale{0.6}

\tikzset{
  root/.style = {circle,scale=\rootscale,fill=\rootcolor},
    rc/.style 2 args = {right=#1*1.5*\clustersep of {#2.east|-first},root}, rr/.style = {right=\rootsep of {#1.east|-first},root},
  roott/.style = {circle,inner sep=-2pt,minimum size=5pt,black,font=\ttfamily\footnotesize},
    rct/.style 2 args = {right=#1*1.5*\clustersep of {#2.east|-first},roott}, rrt/.style = {right=\rootsep of {#1.east|-first},roott},
  rootA/.style = {circle,scale=\rootscaleA,ball color=\rootcolorA},
    rcA/.style 2 args = {right=#1*1.5*\clustersep of {#2.east|-first},rootA}, rrA/.style = {right=\rootsep of {#1.east|-first},rootA},
  rootB/.style = {circle,scale=\rootscaleB,ball color=\rootcolorB},
    rcB/.style 2 args = {right=#1*1.5*\clustersep of {#2.east|-first},rootB}, rrB/.style = {right=\rootsep of {#1.east|-first},rootB},
  rootC/.style = {diamond,scale=\rootscaleC,ball color=\rootcolorC},
    rcC/.style 2 args = {right=#1*1.5*\clustersep of {#2.east|-first},rootC}, rrC/.style = {right=\rootsep of {#1.east|-first},rootC},
  rootD/.style = {circle,scale=\rootscaleD,ball color=\rootcolorD},
    rcD/.style 2 args = {right=#1*1.5*\clustersep of {#2.east|-first},rootD}, rrD/.style = {right=\rootsep of {#1.east|-first},rootD},
  cluster/.style = {draw=black!90,thick,rounded corners,inner sep=22*\clustersep,outer xsep=22*\clustersep,fit=#1},
  clabel/.style  = {anchor=west,scale=\cdepthscale,black,inner sep=0,outer xsep=1,outer ysep=0},
  clabelL/.style = {above right=-\clustersep of #1t.north east,clabel},
  clabelD/.style = {below right=-\clustersep of #1t.south east,clabel},
  clouter/.style = {inner sep=0,outer sep=0,fit=#1}
}

\def\Cluster #1 = #2;{\node[cluster=#2] (#1) {};}
\def\ClusterL #1[#2] = #3;{
  \node[cluster=#3] (#1t) {}; \node[clabelL=#1] (#1l) {$#2$}; \node[clouter=(#1t)(#1l)] (#1) {};}
\def\ClusterD #1[#2] = #3;{
  \node[cluster=#3] (#1t) {}; \node[clabelD=#1] (#1d) {$#2$}; \node[clouter=(#1t)(#1d)] (#1) {};}
\def\ClusterLD #1[#2][#3] = #4;{
  \node[cluster=#4] (#1t) {}; \node[clabelL=#1] (#1l) {$#2$}; 
  \node[clabelD=#1] (#1d) {$#3$}; \node[clouter=(#1t)(#1l)(#1d)] (#1) {};}
\def\ClusterLDName #1[#2][#3][#4] = #5;{
  \node[cluster=#5] (#1t) {}; \node[clabelL=#1] (#1l) {$#2$}; 
  \node[clabelD=#1] (#1d) {$#3$}; 
  \node[scale=\cnamescale,above=\clustersep/3 of #1t,inner sep=0, outer sep=0] (#1n) {$#4$}; 
  \node[clouter=(#1l)(#1d)(#1t)] (#1) {};}

\newcommand{\Root}[4][]{
  \ifx\relax#2\relax\node[rr#1=#3] (#4) {};\else\node[rc#1={#2}{#3}] (#4) {};\fi}
\newcommand{\RootT}[5][]{
  \ifx\relax#2\relax\node[rrt#1=#3] (#4) {#5};\else\node[rct#1={#2}{#3}] (#4) {#5};\fi}

\def\frob(#1)(#2){\path[draw,thick,shorten <=-22*\clustersep,shorten >=-22*\clustersep](#1.east)--(#2.west|-#1){};}

\usepackage{pbox}
\def\pb#1{\pbox[c]{\textwidth}{\hfil #1\hfil}}

\long\def\clusterpicture#1\endclusterpicture{\pb{\vbox to \cltopskip{\vfill}\\%
  \begin{tikzpicture}\node[coordinate] (first) {};#1\end{tikzpicture}\\[-11pt]\vbox to \clbottomskip{\vfill}}}   
\long\def\clusterpictureopt#1#2\endclusterpicture{\pb{\vbox to \cltopskip{\vfill}\\%
  \begin{tikzpicture}[#1]\node[coordinate] (first) {};#2\end{tikzpicture}\\[-11pt]\vbox to \clbottomskip{\vfill}}}

\def\pb#1{\pbox[c]{\textwidth}{\hfil #1\hfil}}

\def\GraphVertices{\SetVertexNormal[Shape=circle, FillColor=blue!50, LineColor=blue!50, LineWidth=0.8pt]
  \tikzset{VertexStyle/.append style = {inner sep=0.5pt,minimum size=0.3em,font = \tiny\bfseries}}}

\def\BlueEdges{  \SetUpEdge[lw=0.8pt,color=blue!70]
   \tikzset{EdgeStyle/.append style = {shorten <=0.5pt,shorten >=0.5pt}}}

\def\LoopW(#1){
  \path[draw,-,thick,color=blue!70] (#1) edge[out=155,in=90] ($(#1)-(1.3,0)$);
  \path[draw,-,thick,color=blue!70] (#1) edge[out=210,in=270] ($(#1)-(1.3,0)$);
}
\def\LoopE(#1){
  \path[draw,-,thick,color=blue!70] (#1) edge[out=25,in=90] ($(#1)+(1.3,0)$);
  \path[draw,-,thick,color=blue!70] (#1) edge[out=-25,in=270] ($(#1)+(1.3,0)$);
}
\def\LoopS(#1){
  \path[draw,-,thick,color=blue!70] (#1) edge[out=115,in=180] ($(#1)+(0,1.2)$);
  \path[draw,-,thick,color=blue!70] (#1) edge[out=65,in=0] ($(#1)+(0,1.2)$);
}
\def\LoopN(#1){
  \path[draw,-,thick,color=blue!70] (#1) edge[out=-115,in=180] ($(#1)-(0,1.2)$);
  \path[draw,-,thick,color=blue!70] (#1) edge[out=-65,in=0] ($(#1)-(0,1.2)$);
}
\def\EdgeW(#1){
  \path[draw,-,thick,color=blue!70] (#1+) edge[out=180,in=90] ($(#1+)-(1.3,0.3)$);
  \path[draw,-,thick,color=blue!70] (#1-) edge[out=180,in=270] ($(#1+)-(1.3,0.3)$);
}
\def\EdgeE(#1){
  \path[draw,-,thick,color=blue!70] (#1+) edge[out=0,in=90] ($(#1-)+(1.3,0.3)$);
  \path[draw,-,thick,color=blue!70] (#1-) edge[out=0,in=270] ($(#1-)+(1.3,0.3)$);
}
\def\EdgeS(#1){
  \path[draw,-,thick,color=blue!70] (#1+) edge[out=90,in=0] ($(#1-)+(0.3,1.3)$);
  \path[draw,-,thick,color=blue!70] (#1-) edge[out=90,in=180] ($(#1-)+(0.3,1.3)$);
}
\def\EdgeN(#1){
  \path[draw,-,thick,color=blue!70] (#1+) edge[out=270,in=0] ($(#1+)-(0.3,1.3)$);
  \path[draw,-,thick,color=blue!70] (#1-) edge[out=270,in=180] ($(#1+)-(0.3,1.3)$);
}
\def\GCircle(#1,#2)(#3,#4){
  \path(#1,#2) node[coordinate] (1) {};
  \path(#3,#4) node[coordinate] (2) {};
  \path[draw,-,thick,color=blue!70] (1) edge[out=90,in=90] (2);
  \path[draw,-,thick,color=blue!70] (2) edge[out=270,in=270] (1);
}

\def\EdgeSign(#1)(#2)#3(#4)#5{
  \node at ($(#1)!#3!(#2) + (#4)$) [color=black, scale=\graphdslabelscale] {$\scriptstyle #5$};
}

\def\GraphEdgeSignDist{0.55}

\def\GraphEdgeSignS(#1)(#2)#3#4{\EdgeSign(#1)(#2)#3(0,-\GraphEdgeSignDist){#4}}

\def\VSwap#1#2#3#4{\path[draw](#1) edge[<->,#3,shorten >=#4pt,shorten <=#4pt] (#2){};}

\def\tgrGVU{\raise-7pt\hbox{\begin{tikzpicture}[scale=\GraphScale]
  \GraphVertices
  \Vertex[x=1.50,y=0.000,L=1]{1};
  \Vertex[x=0.000,y=0.000,L=1]{2};
  \BlueEdges
  \Edge(1)(2)
\GraphEdgeSignS($(1)+(0,0.3)$)(2){0.5}{\frac{1}{2}(m+n)}\end{tikzpicture}}}

\def\tgrGVD{\raise-7pt\hbox{\begin{tikzpicture}[scale=\GraphScale]
  \GraphVertices
  \Vertex[x=1.50,y=0.000,L=1]{1};
  \Vertex[x=0.000,y=0.000,L=1]{2};
  \BlueEdges
  \Edge(1)(2)
\GraphEdgeSignS($(1)+(0,0.3)$)(2){0.5}{\frac{1}{2}n)}\end{tikzpicture}}}

\def\tgrGVT{\raise-7pt\hbox{\begin{tikzpicture}[scale=\GraphScale]
  \GraphVertices
  \Vertex[x=1.50,y=0.000,L=1]{1};
  \Vertex[x=0.000,y=0.000,L=1]{2};
  \BlueEdges
  \Edge(1)(2)
\GraphEdgeSignS($(1)+(0,0.3)$)(2){0.5}{\frac{1}{2}n}\end{tikzpicture}}}

\def\tgrGVQ{\raise-7pt\hbox{\begin{tikzpicture}[scale=\GraphScale]
  \GraphVertices
  \Vertex[x=1.50,y=0.000,L=1]{1};
  \Vertex[x=0.000,y=0.000,L=1]{2};
  \BlueEdges
  \Edge(1)(2)
\GraphEdgeSignS($(1)+(0,0.3)$)(2){0.5}{\frac{1}{2}(m-n)}\end{tikzpicture}}}

\def\tgrGW{\raise-7pt\hbox{\begin{tikzpicture}[scale=\GraphScale]
  \GraphVertices
  \Vertex[x=1.50,y=0.000,L=1]{1};
  \Vertex[x=0.000,y=0.000,L=1]{2};
  \BlueEdges
  \Edge(1)(2)
\GraphEdgeSignS(1)(2){0.5}{r}\VSwap{1}{2}{in=30,out=150}{2}\end{tikzpicture}}}

\title{Computing Euler factors of genus 2 curves\\at odd primes of almost good reduction}\author{C\'eline Maistret, Andrew V. Sutherland}\date{}

\address{School of Mathematics, University of Bristol, Bristol BS8 1UG, UK}
\email{celine.maistret@bristol.ac.uk}
\address{Department of Mathematics, Massachusetts Institute of Technology, Cambridge, MA 02139, USA}
\email{drew@math.mit.edu}

\begin{document}
%%%%%%%%%%%%%%%%%%%%%%%%%%%%%%%%%
%%%%%%%%%%%%%%%%%%%%%%%%%%%%%%%%%
\begin{abstract}
We present an efficient algorithm to compute the Euler factor of a genus 2 curve $C/\Q$ at an odd prime $p$ that is of bad reduction for $C$ but of good reduction for the Jacobian of $C$ (a prime of ``almost good'' reduction).
Our approach is based on the theory of cluster pictures introduced by Dokchitser, Dokchitser, Maistret, and Morgan, which allows us to reduce the problem to a short, explicit computation over $\Z$ and $\F_p$, followed by a point-counting computation on two elliptic curves over $\F_p$, or a single elliptic curve over $\F_{p^2}$.
A key feature of our approach is that we avoid the need to compute a regular model for $C$.  This allows us to efficiently compute many examples that are infeasible to handle using the algorithms currently available in computer algebra systems such as Magma and Pari/GP.
 \end{abstract}
%%%%%%%%%%%%%%%%%%%%%%%%%%%%%%%%%
%%%%%%%%%%%%%%%%%%%%%%%%%%%%%%%%%
\maketitle
%%%%%%%%%%%%%%%%%%%%%%%%%%%%%%%%%
%%%%%%%%%%%%%%%%%%%%%%%%%%%%%%%%%
{\hypersetup{}
% or \hypersetup{linkcolor=black}, if the colorlinks=true option of hyperref is used
\setcounter{tocdepth}{1}
%\tableofcontents}

%%%%%%%%%%%%%%%%%%%%%%%%%%%%%%%%%
%%%%%%%%%%%%%%%%%%%%%%%%%%%%%%%%%
\section{Introduction}
The $L$-functions and modular forms database (LMFDB) \cite{LMFDB} currently contains a database of 66\,158 genus 2 curves $C/\Q$ with absolute discriminant bounded by $10^6$ whose computation is described in \cite{BSSVY}.  The bound on the discriminant is motivated by the fact that it also serves as a bound on the conductor of the $L$-function of $C$.
In the context of the LMFDB one naturally organizes objects according to the conductor of their $L$-function, since this invariant determines the level of the corresponding modular form predicted by the Langlands program (in the case of genus 2 curves, a Siegel modular form that is generically a paramodular form).  One prioritizes examples of small conductor, because this makes it more feasible to compute their $L$-functions, and to identify other objects in the LMFDB (including modular forms) that have the same $L$-function.

But bounding the discriminant is overly restrictive: genus 2 curves of small conductor need not have small discriminant.  This is also true of elliptic curves, but becomes much more pronounced in higher genus due to the possibility that the minimal discriminant of a curve may be divisible by primes $p$ that do not divide the conductor of its $L$-function, which cannot happen in genus~1.  This necessarily happens when the Jacobian of $C$, reduces to a product of elliptic curves over $\F_p$, which forces the curve to have bad reduction.  This is not an uncommon occurrence, and may happen even when the Jacobian is geometrically simple.  We call a prime of bad reduction for $C$ that does not divide the conductor of its $L$-function a prime of \defi{almost good reduction} for $C$.

There is work in progress to expand the genus 2 curve database in the LMFDB to include more than five million genus 2 curves over $\Q$ with conductor below $2^{20}\approx 10^6$, and more than half of these curves have almost good reduction at some prime, including almost half a million that have geometrically simple Jacobians \cite{S24}.  The primes $p$ involved need not be small and may substantially exceed the conductor bound.
For example, the curve

\begin{footnotesize}
\begin{align*}
C:y^2 + (x^3 + x^2 + x)y = &-144061786290072x^6 - 23062462482396x^5 - 1266273619292236x^4 - 3052943051575761x^3\\
& + 3989955132045666x^2 + 50048078951052415x - 24854569174209566
\end{align*}
\end{footnotesize}

\noindent
has prime conductor $270761$ and minimal discriminant $270761\cdot 14556001^{22}$, with $p=14\,556\,001$ a prime of almost good reduction, and
there are examples with conductor less than $2^{20}$ that have primes of almost good reduction as large as $43\,858\,540\,753$.

None of the computer algebra systems Magma \cite{magma}, Pari/GP \cite{pari}, or SageMath \cite{sage} is currently able to compute the Euler factor at $p=14\,556\,001$ for the $L$-function of the curve $C$ listed above, and the method for computing bad Euler factors using the approximate functional equation described in~\cite{BSSVY} cannot be feasibly applied here.  In this example $p$ is so much larger than the conductor that the inability to compute the Euler factor is not critical (one can approximate special values and bound the analytic rank without it), but there are hundreds of thousands of cases with $p$ close to the square root of the conductor, a regime where it is difficult to accurately approximate the $L$-function without knowing the Euler factor at $p$ and also difficult to guess the correct Euler factor and heuristically check it using the functional equation. Moreover, a single genus 2 curve may have multiple primes of almost good reduction, which complicates matters further.  There are several curves in the new dataset that have five or more primes of almost good reduction, all of which are small enough to have a major impact on the functional equation, and the number of possibilities for the bad Euler factors is far too large to make heuristic checking feasible.
These almost good primes have proven to be a major obstacle to attempts to expand the database of genus 2 curves in the LMFDB, which is what motivated the work we present here.  Our main result is the following.

\begin{theorem}
Let $C/\Q$ be a genus $2$ curve $y^2=f(x)=\sum_i f_ix^i\in \Z[x]$ with almost good reduction at an odd prime $p$.  There is a deterministic algorithm that, given a nonsquare element of $\F_p^\times$, computes the $L$-polynomial $L_p(C,T)$ in time
\[
O(\|f\|^2\log^2\!\|f\|/\log p + \log^5\!p),
\]
where $\|f\|=\max_i\log\|f_i|$.  There is also a Las Vegas algorithm with the same expected running time that does not require a nonsquare element of $\F_p^\times$ to be provided as part of the input.
\end{theorem}

For large $p$ the running time of this algorithm is dominated by the time to count points on elliptic curves using Schoof's algorithm.  When $p$ is small and the power of $p$ dividing the discriminant is bounded, the running time is quasi-linear in $\|f\|$, which is the size of the input.
The algorithm is described in detail in Section~\ref{sec:alg}, and it is easy to implement.  A simple implementation of the new algorithm in Magma is available in the GitHub repository associated to this paper \cite{repo}, and it is already a dramatic improvement over the \textsc{EulerFactor} function implemented in Magma.\footnote{We compare our algorithm to the \textsc{EulerFactor} intrinsic in Magma because it is the only implementation we are aware of that gives correct results when the computation terminates without reporting an error; neither Pari/GP nor SageMath currently support the computation of Euler factors for genus 2 curves at primes of almost good reduction.} It took roughly 242 CPU days to compute Euler factors at approximately 3.5 million primes of almost good reduction arising in our small conductor dataset using the \textsc{EulerFactor} function, except for~489 cases where the computation did not terminate within 8 hours, an average of about 6 seconds per Euler factor.  By contrast, a simple Magma implementation of our new algorithm takes less than 1.3 CPU hours to compute every Euler factor, averaging close to one millisecond per Euler factor with a maximum time of less than 25 milliseconds, including the 489 examples we were not able to compute using the \textsc{EulerFactor} intrinsic.  A low level C implementation of the new algorithm takes only 30 CPU seconds to accomplish the same task (about 8 microseconds per Euler factor).  See Section~\ref{sec:implementation} and Tables~\ref{tab:summary}-\ref{tab:fail} for further details of our implementation and the tests we ran.

Our algorithm does not treat curves with almost good reduction at the prime $p=2$, even though there are many such examples, notably including the modular curve $X_0(22)$ whose Jacobian has conductor $11^2$, the smallest conductor possible for an abelian surface.  But our algorithm is still helpful in treating these cases because it makes it feasible to apply the methods of \cite{BSSVY} that can efficiently use the functional equation to compute the Euler factor at 2, provided the Euler factors at all larger primes are known.

We expect that the cluster picture approach we use here can be extended to other types of bad reduction for genus 2 curves, at least in a semistable setting; this is an area for future work.  We should note that while our algorithms could conceivably be extended to handle almost good primes of genus 3 hyperelliptic curves, our method does not scale well with the genus.  There is recent work on new methods for computing regular models, notably that of Dokchitser \cite{TimD} and Muselli \cite{Muselli} that is likely to handle higher genus curves better than our approach, and may also allow one to treat other reduction types, with similar efficiency.

\subsection{Acknowledgments} We are grateful to Matthew Bisatt and Tim Dokchitser for helpful conversations.  Maistret was supported by a Royal Society Dorothy Hodgkins Fellowship, and Sutherland was supported by Simons Foundation grant 550033.

\section{Background}

\subsection{Euler factors}
Let $C$ be a smooth projective curve of genus $g \ge 2$ over $\Q$. The \defi{$L$-function} of~$C$ is defined as the Euler product 
$$
L(C,s) = \prod_p L_p(C,p^{-s})^{-1},
$$
where the \defi{$L$-polynomial} $L_p\in \Z[T]$ is given by 
$$
L_p(C,T)\coloneqq\det \left(1-T\,\Frob_p^{-1}|H_{\text{\'et}}^1(C\otimes_{\Q}\Q^{\text{alg}},\Q_\ell)^{I_p}\right),
$$ 
and has degree $2g$ at primes of good reduction for the Jacobian of $C$, including primes of almost good reduction for $C$, and otherwise has degree strictly less than $2g$.

Here $\Frob_p$ denotes an arithmetic Frobenius element, $I_p$ is the inertia subgroup of a decomposition group $D_p \subset \Gal(\Q^{\text{alg}}/\Q)$ and $\ell$ is a prime distinct from $p$. 

\begin{proposition}\label{pr:SF}
Let $C/\Q$ be a genus $2$ curve and let $p$ be a prime of almost good reduction for~$C$. Let $\mathcal{C}$ be a $\Z_p$-model which is regular, and let $\bar{C}$ denote its special fiber. Then 
$$
H_{\textup{\'et}}^1(C\otimes_{\Q}\Q^{\textup{alg}},\Q_\ell)^{I_p} \cong H_{\textup{\'et}}^1(\bar{C}_{\overline{\F}_p},\Q_\ell).
$$
\end{proposition}
\begin{proof}
Note that all the components of $\bar{C}$ have multiplicity 1, since $C$ has semistable reduction at~$p$. The proposition then follows directly from Proposition 2.8 in \cite{BW}.
\end{proof}
One can compute the Euler factor $L_p(C,T)$ via its connection to the zeta function of $\bar{C}/\F_p$, although we note that this is not the only possible approach; see Section 1.1 of  \cite{BW}.
\begin{definition} Let $X/\F_p$ be a curve. We define its \defi{zeta function} by 
$$
Z(X/\F_p,T) = \exp\left(\sum_{n\ge1}|X(\F_{p^n})|\frac{T^n}{n}\right).
$$
\end{definition}

\begin{theorem}\label{th:ZetaEuler} Let $\bar{C}/\F_p$ be the special fiber of a $\Z_p$-regular model of $C/\Q$. Then 
$$
Z(\bar{C},T) =\frac{P_1(T)}{P_0(T)P_2(T)}, 
$$
where $P_i(T)=\det (1-T\Frob_p^{-1}|H_{\text{\'et}}^i(\bar{C}_{\overline{\F}_p},\Q_\ell)$ for $i=0,1,2$ and some prime $\ell \ne p$.
In particular, $P_1(T)$ is the $L$-polynomial $L_p(C,T)$.
\end{theorem}
\begin{proof} This is Theorem 13.1 in \cite{MilneEtale} with $X=\bar{C}$, combined with Proposition~\ref{pr:SF}.
\end{proof}
It follows from Theorem~\ref{th:ZetaEuler} that one can compute $L_p(C,T)$ by computing the zeta function of the special fiber of a $\Z_p$-regular model of $C/\Q_p$.  However, this may be computationally expensive; indeed, in the 489 cases mentioned in the introduction where Magma struggled to compute Euler factors at primes of almost good reduction, the difficulties arose while computing a regular model.

We will use the theory of cluster pictures to avoid computing a regular model.
For a general exposition of the theory of cluster pictures, we refer the interested reader to \cite{hyper}, where explicit examples are used to illustrate theoretical points. In particular, \cite[Section 6]{hyper} provides examples of construction of special fibers from cluster pictures. 

\subsection{Cluster Pictures}
Let $K$ be a local field of odd residue characteristic $p$, let $v$ be a normalized valuation with respect to $K$, and let $C/K$ a hyperelliptic curve of genus $g\ge2$ given by 
$$
y^2 = f(x) = c \prod_{r \in \cR}(x-r),
$$
where $f \in K[x]$ is separable with $\deg(f) = 2g+1$ or $2g+2$, and where $\cR$ denotes the set of roots of $f(x)$ in $K^{\text{sep}}$.
\begin{definition}[Clusters and cluster pictures]
\label{def:cluster}
A \defi{cluster} is a non-empty subset $\s \subseteq \cR$ of the form $\s = D \cap \cR$ for some disc $D=\{x\!\in\! K^{\textup{alg}} \mid v(x-z)\!\geq\! d\}$ for some $z\in K^{\textup{alg}}$ and $d\in \Q$. 

For a cluster $\s$ with $|\s|>1$, its \defi{depth} $d_\s$ is the maximal $d$ for which $\s$ is cut out by such a disc, that is $d_\s\! =\! \min_{r,r' \in \mathfrak{s}} v(r\!-\!r')$. If moreover $\s\neq \cR$, we define the \defi{parent} cluster $P(\s)$ of $\s$ to be the smallest cluster with $\s\subsetneq P(\s)$. Then the \defi{relative depth} of $\s$ is $\delta_\s\! =\! d_\s\! -\!d_{P(\s)}$.

We refer to this data as the \defi{cluster picture} of $C$.
\end{definition}

\begin{remark}\label{CLrem1}
The absolute Galois group of $K$ acts on clusters via its action on the roots. This action preserves depths and containments of clusters. 
\end{remark}

\begin{definition}

If $\s'\subsetneq \s$ is a maximal subcluster, we 
refer to $\s'$ as a \defi{child} of $\s$.

For two clusters (or roots)  $\s_1$, $\s_2$ write $\s_1\wedge \s_2$ for the smallest cluster containing them.  
\end{definition}

\begin{definition} \label[definition]{principaldefi}

A cluster $\s$ is \defi{principal} except when:
\begin{itemize}
\item 
$|\s|\le 2$, or
\item
$\s$ has a child of size $2g$, or
\item 
$\s=\cR$ is even and has exactly two children.
\end{itemize}
\end{definition}

\begin{definition}\label[definition]{no:nu}
For a cluster $\s$ set
$$
 \nu_\s = v(c)+\sum_{r\in\cR} d_{r\wedge \s}.
$$  
\end{definition}

\begin{notation}\label{CLnotation}
We draw cluster pictures by drawing roots $r \in \cR$ as \smash{\raise4pt\hbox{\clusterpicture\Root[A]{1}{first}{r1};\endclusterpicture}}, and we draw ovals around roots to represent clusters (of size $>1$), such as:
$$
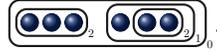
\scalebox{1}{\clusterpicture            % [[[1,2]_{n},3,4,5],6]
  \Root[A] {1} {first} {r1};
  \Root[A] {} {r1} {r2};
  \Root[A] {} {r2} {r3};
  \Root[A] {4} {r3} {r4};
  \Root[A] {1} {r4} {r5};
  \Root[A] {} {r5} {r6};
  \ClusterLD c1[][{2}] = (r1)(r2)(r3);
  \ClusterLD c2[][{2}] = (r5)(r6);
  \ClusterLD c3[][{1}] = (r4)(c2);
  \ClusterLD c4[][{0}] = (c1)(c2)(c3);
\endclusterpicture}.
$$
The subscript on the largest cluster $\cR$ is its depth, while the subscripts on the other clusters are relative depths.
\end{notation}

\begin{example} Consider the genus 2 curve $C:y^2 = x(x-p^2)(x-3p^2)(x-1)(x-1+p^4)(x-1-p^4)$ where $p \ge 5$ is a prime. Its cluster picture at $p$ is given by
 $$
\scalebox{1}{\clusterpicture            % [[[1,2]_{n},3,4,5],6]
  \Root[A] {1} {first} {r1};
  \Root[A] {} {r1} {r2};
  \Root[A] {} {r2} {r3};
  \Root[A] {4} {r3} {r4};
  \Root[A] {} {r4} {r5};
  \Root[A] {} {r5} {r6};
  \ClusterLD c1[][{2}] = (r1)(r2)(r3);
  \ClusterLD c2[][{4}] = (r4)(r5)(r6);
  \ClusterLD c3[][{0}] = (c1)(c2);
\endclusterpicture}
$$

\end{example}

\section{Cluster pictures and special fibers}

\begin{proposition}\label{pr:E1E2}
Let  $C/\Q$ be a genus $2$ curve and $p$ an odd prime of almost good reduction for $C$. Then 
the possible cluster pictures for $C$ at $p$ are 
$$
\scalebox{.85}{
\textbf{\textup{1)}} \clusterpicture            % [[[1,2]_{n},3,4,5],6]
  \Root[A] {1} {first} {r1};
  \Root[A] {} {r1} {r2};
  \Root[A] {} {r2} {r3};
  \Root[A] {4} {r3} {r4};
  \Root[A] {} {r4} {r5};
  \Root[A] {} {r5} {r6};
  %\ClusterLD c1[][{2}] = (r1)(r2);
  \ClusterLD c2[][{n}] = (r4)(r5)(r6);
  \ClusterLD c3[][{t}] = (r1)(r2)(r3)(c2);
 \endclusterpicture
 ,\qquad 
\textbf{\textup{2)}} \clusterpicture            % [[[1,2]_{n},3,4,5],6]
  \Root[A] {1} {first} {r1};
  \Root[A] {} {r1} {r2};
  \Root[A] {} {r2} {r3};
  \Root[A] {6} {r3} {r4};
  \Root[A] {} {r4} {r5};
  \Root[A] {} {r5} {r6};
  \ClusterLD c1[][{m}] = (r1)(r2)(r3);
  \ClusterLD c2[][{n}] = (r4)(r5)(r6);
  \ClusterLD c3[][{t}] = (c1)(c2);
 \endclusterpicture
,\qquad
\textbf{\textup{3)}} \clusterpicture            % [[[1,2]_{n},3,4,5],6]
  \Root[A] {1} {first} {r1};
  \Root[A] {} {r1} {r2};
  \Root[A] {4} {r2} {r3};
  \Root[A] {} {r3} {r4};
  \Root[A] {} {r4} {r5};
  %\ClusterLD c1[][{2}] = (r1)(r2);
  \ClusterLD c2[][{n}] = (r3)(r4)(r5);
  \ClusterLD c3[][{t}] = (r1)(r2)(c2);
 \endclusterpicture
 ,\qquad 
\textbf{\textup{4)}} \clusterpicture            % [[[1,2]_{n},3,4,5],6]
  \Root[A] {1} {first} {r1};
  \Root[A] {4} {r1} {r2};
  \Root[A] {} {r2} {r3};
  \Root[A] {4} {r3} {r4};
  \Root[A] {} {r4} {r5};
  \Root[A] {} {r5} {r6};
  \ClusterLD c2[][{m-n}] = (r4)(r5)(r6);
   \ClusterLD c3[][{n}] = (r2)(r3)(c2);
  \ClusterLD c4[][{t}] = (r1)(c3)(c2);
 \endclusterpicture
 },
$$
for some integers $m, n, t$, with $m\ge n$ in case 2) and $m>n$ in case 4), such that the following congruence relations are satisfied:\\
\indent $v_p(c)$ and $n$ are even in cases 1),\\
\indent $v_p(c) \equiv n \equiv m \bmod 2$ in cases 2) and 4),\\
\indent $v_p(c) \equiv t  \bmod 2$ and $n$ is even in cases 3).  

\end{proposition}

\begin{proof}
This is Theorem 10.3 (5) in \cite{m2d2} which implies that clusters $\s \ne \cR$ in the picture must contain an odd number of roots. This gives the classification of pictures above. The parity of the valuation of the leading term and the depths follow from the fact that $\nu_\s \in 2\Z$ for principal clusters. 
\end{proof}

\begin{remark}\label{rk:unram}
\phantom{gimme a newline}
\begin{enumerate}[1.]
\setlength{\itemsep}{6pt}

\item The assumption that $p$ is a prime of almost good reduction for $C$ is essential in Proposition 3.1, as it controls the parity of $v_p(c)$.  Changing the parity of $v_p(c)$ amounts to taking a quadratic twist, which does not change the cluster picture but will force $\Jac C$ to have bad reduction at $p$.  Our algorithms will produce the same output given $C$ or its quadratic twist by $p$, but in the latter case the input will not satisfy the necessary assumption that the Jacobian has good reduction at~$p$; our algorithms cannot be used to compute the $L$-polynomial of the quadratic twist.

\item The assumption that $p$ is a prime of good reduction for $\Jac C$ implies that the splitting field $L$ of $f$ is unramified (Theorem 10.3 (5) in \cite{m2d2}).
It follows that $p$ is a uniformizer for $L$ and that the normalized valuation $v$ extends the valuation $v_p$ of $K=\Q_p$ with index 1.

\item The reduction type for $C$ in Proposition \ref{pr:E1E2} corresponds to Namikawa-Ueno type $I_0-I_0-r$ listed in \cite{m2d2} Table 18.2.  One can check using Theorem 8.6 in \cite{m2d2} that the dual graphs of the special fibers corresponding to the four cluster pictures are
$$
 \textbf{\textup{1)}}  \tgrGVT, \quad \textbf{2)}  \tgrGVU, \quad \textbf{3)}  \tgrGVT, \quad \textbf{4)}  \tgrGVQ.  
 $$
 In the drawings above, vertices represent genus 1 components. These are linked by chains of $\frac{1}{2}n, \frac{1}{2}(m+n), \frac{1}{2}n$, and $\frac{1}{2}(m-n)$ edges respectively, representing genus 0 components in the special fiber. 
\item When $n=m$ in picture \textbf{\textup{2)}} of Proposition \ref{pr:E1E2}, the Frobenius automorphism will permute the two clusters of size 3 if their centers are defined over the unramified quadratic extension of $\Q_p$ (and therefore are Galois conjugates). If this is the case we will draw a line between the two clusters in the picture as so:  
 $ \clusterpicture            % [[[1,2]_{n},3,4,5],6]
  \Root[A] {1} {first} {r1};
  \Root[A] {} {r1} {r2};
  \Root[A] {} {r2} {r3};
  \Root[A] {4} {r3} {r4};
  \Root[A] {} {r4} {r5};
  \Root[A] {} {r5} {r6};
  \ClusterLD c1[][{n}] = (r1)(r2)(r3);
  \ClusterLD c2[][{n}] = (r4)(r5)(r6);
  \frob(c1)(c2);
  \ClusterLD c3[][{}] = (c1)(c2);
 \endclusterpicture.
 $ Note that in this case, the depths are necessarily equal. 
\end{enumerate}

\end{remark}

\begin{remark}\label{rem:normalize}
\phantom{gimme a newline}
\begin{enumerate}[1.]
\setlength{\itemsep}{6pt}

%The depth $d_{\cR}$ of the outer cluster in Proposition \ref{pr:E1E2} as it can be arbitrary. However, the first step of our algorithm will ensure that $d_{\cR}=0$. 
 
\item We can assume without loss of generality that $C\colon y^2=f(x)$ is defined by a polynomial $f\in \Z[x]$ of degree 6. This will ensure that the cluster picture of $C$ is never of type $\textbf{\textup {3)}}$.
 
 \item The depth $d_{\cR}=t$ of the outer cluster in types $\textbf{\textup {1)}}$, $\textbf{\textup {2)}}$, $\textbf{\textup {4)}}$ of Proposition \ref{pr:E1E2} can be set to $t=0$. This is the first step of our algorithm (see Definition \ref{def:pnorm} and Corollary \ref{co:types} below). 

\item We can assume without loss of generality that $v_p(c)$ is $0$ or $1$, and in the latter case that every coefficient of $f\in \Z[x]$ is divisible by $p$.
\end{enumerate}
\end{remark}

We will give a constructive proof of the claims made in Remark~\ref{rem:normalize} when we present our algorithms in the next section, but they motivate the following definition.

\begin{definition}\label{def:pnorm}
Let $p$ be an odd prime. A squarefree sextic polynomial $f=\sum_i f_ix^i\in \Z[x]$ defining a genus 2 curve $C\colon y^2=f(x)$ with almost good reduction at $p$ that  satisfies $d_{\cR}=t=0$ and $v_p(f_6)=\min_i\{v_p(f_i)\}\le 1$ is said to be \defi{$p$-normalized}.
\end{definition}

%{\color{blue}
%{\color{red}\st{Argue somewhere that in each cluster there is at least a $\Q_p$-rational root}} (Semistability implies ramification index of splitting field at most 2 so inertia cannot cyclically permute 3 roots). Drew: I think instead we want to argue that the splitting field of $f$ is unramified, which I have verified is true for all the examples we have (we could get away with proving that that each cluster contains a root in an unramified extension, but we would then need to adjust the statement of Corollaries 3.10 and 3.11).}

\begin{corollary}\label{co:types}
Let $p$ be an odd prime and let $C\colon y^2=f(x)$ be a genus 2 curve defined by a $p$-normalized polynomial $f\in \Z[x]$.
Then the cluster picture for $C$ is one of the following:\smallskip

Type \textbf{1} := 
$\clusterpicture            % [[[1,2]_{n},3,4,5],6]
  \Root[A] {1} {first} {r1};
  \Root[A] {} {r1} {r2};
  \Root[A] {} {r2} {r3};
  \Root[A] {4} {r3} {r4};
  \Root[A] {} {r4} {r5};
  \Root[A] {} {r5} {r6};
  %\ClusterLD c1[][{2}] = (r1)(r2);
  \ClusterLD c2[][{n}] = (r4)(r5)(r6);
  \ClusterLD c3[][{0}] = (r1)(r2)(r3)(c2);
 \endclusterpicture$, where $v_p(c)=0\equiv n\bmod 2$,
 
 Type \textbf{2a} := 
 $
  \clusterpicture            % [[[1,2]_{n},3,4,5],6]
  \Root[A] {1} {first} {r1};
  \Root[A] {} {r1} {r2};
  \Root[A] {} {r2} {r3};
  \Root[A] {4} {r3} {r4};
  \Root[A] {} {r4} {r5};
  \Root[A] {} {r5} {r6};
  \ClusterLD c1[][{m}] = (r1)(r2)(r3);
  \ClusterLD c2[][{n}] = (r4)(r5)(r6);
  \ClusterLD c3[][{0}] = (c1)(c2);
 \endclusterpicture
 $, where $m\ge n$ and $v_p(c) \equiv m \equiv n \bmod 2$ with $v_p(c)\le 1$,
 
 Type  \textbf{2b} := 
 $
  \clusterpicture            % [[[1,2]_{n},3,4,5],6]
  \Root[A] {1} {first} {r1};
  \Root[A] {} {r1} {r2};
  \Root[A] {} {r2} {r3};
  \Root[A] {4} {r3} {r4};
  \Root[A] {} {r4} {r5};
  \Root[A] {} {r5} {r6};
  \ClusterLD c1[][{n}] = (r1)(r2)(r3);
  \ClusterLD c2[][{n}] = (r4)(r5)(r6);
  \ClusterLD c3[][{0}] = (c1)(c2);
  \frob(c1)(c2);
 \endclusterpicture
 $, where $v_p(c) \equiv n \bmod 2$ with $v_p(c)\le 1$,
 
 Type \textbf{4} := 
 $
 \clusterpicture            % [[[1,2]_{n},3,4,5],6]
  \Root[A] {1} {first} {r1};
  \Root[A] {4} {r1} {r2};
  \Root[A] {} {r2} {r3};
  \Root[A] {4} {r3} {r4};
  \Root[A] {} {r4} {r5};
  \Root[A] {} {r5} {r6};
  \ClusterLD c2[][{m-n}] = (r4)(r5)(r6);
   \ClusterLD c3[][{n}] = (r2)(r3)(c2);
  \ClusterLD c4[][{0}] = (r1)(c3)(c2);
 \endclusterpicture
 $, where $m>n$ and $v_p(c) \equiv m\equiv n \bmod 2$ with $v_p(c)\le 1$. 
\end{corollary}

\begin{proposition}\label{pr:split}
Let $p$ be an odd prime and let $C\colon y^2=f(x)$ be a genus $2$ curve defined by a $p$-normalized polynomial $f\in \Z[x]$, let $\tilde f = p^{-v_p(c)}f\in \Z[x]$, let $\bar f = \tilde f \bmod p \in \F_p[x]$, and let $\bar c =cp^{-v_p(c)}\bmod p\in \F_p$.  Then depending on the type of the cluster picture of $C$ at $p$, exactly one of the following holds:
\begin{itemize}
\setlength{\itemsep}{3pt}
\item[\textbf{1})] $\bar f = \bar c\,(x-\bar r)^3\bar u$ for some squarefree monic cubic $\bar u\in \F_p[x]$ and $\bar r\in \F_p$ with $\bar u(\bar r)\ne 0$.
\item[\textbf{2a})] $\bar f = \bar c\,(x-\bar r)^3(x-\bar s)^3$ for some distinct $\bar r,\bar s\in \F_p$.
\item[\textbf{2b})] $\bar f = \bar c\,\bar u^3$ for some irreducible monic quadratic $\bar u\in \F_p[x]$.
\item[\textbf{4})] $\bar f = \bar c\,(x-\bar r)^5(x-\bar s)$ for some distinct $\bar r,\bar s\in \F_p$.
\end{itemize}
\end{proposition}
\begin{proof}  Let $\alpha_1, \alpha_2, \alpha_3,\beta_1, \beta_2, \beta_3$ denote the roots of $f(x)$ ordered left to right in the cluster picture.

Types \textbf{1}: Here $\beta_1, \beta_2, \beta_3$ denote the roots in the cluster of depth $n$. Since $v(\beta_i-\beta_j) = n$ for $i \ne j = 1,2,3$, we have that $\beta_1 \equiv \beta_2 \equiv \beta_3 \bmod p $, which gives the factor $(x-\bar{r})^3$. Now $v(\alpha_i-\alpha_j) =0$ so that $\alpha_i \not\equiv \alpha_j \bmod p$ for $i\ne j=1,2,3$. Moreover, $v(\alpha_i-\beta_j)=0$ for $i,j = 1,2,3$, therefore $\alpha_i \not \equiv \beta_j$ for  $i,j = 1,2,3$. It follows that $\bar{u}=(x-\overline{\alpha_1})(x-\overline{\alpha_2})(x-\overline{\alpha_3})$ is a square free cubic in $\F_p[x]$ with $\bar{u}(\bar{r})\ne 0$. Types \textbf{2a} and \textbf{4} follow from a similar argument. 

Type \textbf{2b}:
Here the Frobenius automorphism permutes the $\alpha$s and $\beta$s pairwise. In particular, it has order 2 so that $\tilde{f}$ splits into two cubic polynomials $f_1,f_2$ over the unramified quadratic extension $F$ of $\Q_p$.
 %(the extension must be unramified since the curve is semistable).
From the cluster picture we have that $\alpha_1 \equiv \alpha_2\equiv \alpha_3 \bmod p\mathcal{O}_F$ and similarly for the $\beta$s.  
It follows that $\tilde{f}(x) \equiv (x-\lambda)^3(x+\lambda)^3 \bmod p^2$ for some $\lambda \in \F_{p^2}$. Let $\Lambda \in \F_p$ with $\lambda^2 =-\Lambda$ so that $\tilde{f}(x) \equiv (x^2+\Lambda)^3 \bmod p$.
\end{proof}

%{\color{red} We can delete Proposition~\ref{pr:2biff} below (and its proof) once we are happy with Proposition~\ref{pr:split}}
%\begin{proposition}\label{pr:2biff} Let $C/\Q:y^2=f(x)$ be as in Proposition \ref{pr:E1E2} with $f(x)= \sum_i f_i x^i \in \Z[x]$ a squarefree degree 6 polynomial with $v_p(f_6) = \min\{v_p(f_i)\} \le 1$ and cluster picture of type (2) of Proposition  \ref{pr:E1E2}.Then the cluster picture of $C$ is of type $\textbf{2b}$ if and only if $f$ has no $\F_p$-rational roots.
%\end{proposition}
%
%\begin{proof}
%Recall first that $\Jac C$ has good reduction and hence that the roots of $f(x)$ are defined over an unramified extension of $\Q_p$. Let $\s_1$ and $\s_2$ denote the two clusters of size 3 in the picture of $C$. If it is of type $\textbf{2b}$ then the Frobenius automorphism permutes $\s_1$ and $\s_2$, and hence their roots. Therefore none of the roots are $\F_p$-rational. 
%
%Now assume that $f(x)$ has no $\F_p$-rational roots. Since it is defined over $\Q_p$, the Frobenius automorphism must permute the roots while fixing $\cR$. In this case, either it sends the roots of $\s_1$ to those of $\s_2$ pairwise, therefore permuting $\s_1$ and $\s_2$ and we are done; or it permutes cyclically the roots inside $\s_1$ and $\s_2$, hence fixing each of them. The latter is a contradiction since the roots inside each cluster reduce to the same element in $\F_p$. 
%\end{proof}
\begin{corollary}\label{co:E1E2}
Let  $C/\Q$ be a genus $2$ curve and $p$ an odd prime of almost good reduction for $C$. Then 
the special fiber $\bar{C}/\F_p$ of the minimal regular model of $C/\Q_p$ consists of a union of two elliptic curves linked by a chain of $\mathbb{P}^1s$.
\end{corollary}
\begin{proof}
This follows directly from the list of possible cluster pictures given in Proposition \ref{pr:E1E2} and \cite[Definition 8.5 and Theorem 8.6]{m2d2}, where explicit models for the components of the special fibers are constructed. One can then check that the components associated to principal clusters are indeed elliptic curves (see the proof of Corollary \ref{co:E1E2a} for an example of such a construction). Alternatively, since $\Jac C$ has good reduction, it follows from the work of Raynaud (\cite{Raynaud}, \cite[Section 9]{SGA}) that its N\'eron model is an abelian scheme. Since the identity component of the special fiber of the N\'eron model is $\text{Pic}^0 \bar{C}$, this forces $\bar C$ to be a genus 2 curve or a union of two elliptic curves. The former is a contradiction to the prime $p$ being of bad reduction for $C$.  
\end{proof}

\begin{remark}\label{rk:GaloisConj}
For cluster pictures of type \textbf{2b} in Corollary \ref{co:types}, the elliptic curves are defined over~$\F_{p^2}$ and permuted by Frobenius, hence have the same $L$-polynomial.
\end{remark}

Let $E_1$ and $E_2$ be the two elliptic curves in $\bar{C}$ given by Corollary \ref{co:E1E2}. 
We now show how to compute $L_p(C,T)$ from $L_p(E_1,T)$ and $L_p(E_2,T)$.
%\begin{remark}
%\cite{m2d2}[Definition 8.5 and Theorem 8.6] provide an explicit model for both elliptic curves from the cluster picture. Given a curve $C$ as above, it is therefore enough to identify which cluster picture among 1), 2), 3) or 4) corresponds to $C$ and use Theorem 8.6 to compute the models. 
%\end{remark}

\begin{proposition}\label{pr:Lp}
Let  $C/\Q$ be a genus $2$ curve, $p$ an odd prime of almost good reduction for $C$, and $r \ge 0$ an integer. Let $E_1$ and $E_2$ be the two elliptic curves linked by a chain of $\mathbb{P}^1$s of length $r$, whose union forms the special fiber of the minimal regular model of $C/\Q_p$ as given by Corollary~\ref{co:E1E2}. Then 

\begin{enumerate}
\item $L_p(C,T) = L_p(E_1,T)L_p(E_2,T)$ if both $E_1$ and $E_2$ are defined over $\F_p$.
\item $L_p(C,T) = L_p(E_1/\F_{p^2},T^2) = L_p(E_2/\F_{p^2},T^2) $ if both $E_1$ and $E_2$ are defined over $\F_{p^2}$.
\end{enumerate}
\end{proposition}

\begin{proof}
(1) Suppose that $E_1$ and $E_2$ are defined over $\F_p$. It follows that each individual $\mathbb{P}^1$ in the chain linking $E_1$ to $E_2$ is also defined over $\F_p$.
Therefore $P_2(T) = (1-pT)^{2+r}$, by definition, since each component of $\bar{C}$ contributes a factor $(1-pT)$.

Let $\alpha_1, \beta_1, \alpha_2, \beta_2$ be algebraic numbers such that $L_p(E_i,T) =(1-\alpha_iT)(1-\beta_iT)$ for $i=1,2$. In particular, for all $n \in \Z_{>0}, |E_i(\F_{p^n})| = p^n+1-\alpha_i^n-\beta_i^n$. It follows that 
$$
|\bar{C}(\F_{p^n})| = |E_1(\F_{p^n})|+|E_2(\F_{p^n})|+r|\mathbb{P}_1(\F_{p^n})|-(r+1),
$$
where we removed $r+1$ intersection points that were counted twice. Therefore
\begin{align*}
|\bar{C}(\F_{p^n})|&=(2+r)(p^n+1)-\alpha_1^n-\beta_1^n-\alpha_2^n-\beta_2^n-(r+1)\\
&= (2+r)p^n-\alpha_1^n-\beta_1^n-\alpha_2^n-\beta_2^n+1,
\end{align*}
and we have
\begin{align*}
Z(\bar{C}/\F_p,T) 
&= \exp\left(\sum_{n\ge1}((2+r)p^n+1-\alpha_1^n-\beta_1^n-\alpha_2^n-\beta_2^n)\frac{T^n}{n}\right)\\
&=\frac{(1-\alpha_1T)(1-\beta_1T)(1-\alpha_2T)(1-\beta_2T)}{(1-T)(1-pT)^{2+r}} = \frac{L_p(E_1,T)L_p(E_2,T)}{(1-T)(1-pT)^{2+r}}.
\end{align*}
The result then follows from Theorem \ref{th:ZetaEuler}.\\
(2)  Suppose that $E_1$ and $E_2$ are defined over $\F_{p^2}$. We have two cases. \\
\indent (i)  If $r$ is even, then all $\mathbb{P}^1$s in the chain linking $E_1$ and $E_2$ are defined over $\F_{p^2}$ and are permuted by Frobenius. Therefore
 %so that $\Frob(E_1)=E2$  \textcolor{red}{Clarify}.
$P_2(T) = (1-pT)^{\frac{2+r}{2}}(1+pT)^{\frac{2+r}{2}}$.

Let $\alpha, \beta$ be algebraic numbers such that $L_p(E_1,T) =(1-\alpha T)(1-\beta T)$. In particular,
\[
|\bar{C}(\F_{p^n})| = 
\begin{cases}
1, \text{ if } n \text{ is odd,}\\
(2+r)p^n+1-2\alpha^n-2\beta^n, \text{ if } n \text{ is even.}\\
\end{cases}
\]
If we put $m=(2+r)/2$, we then have
\begin{align*}
Z(\bar{C}/\F_p,T) &=
\exp\left(\sum_{n\ge1}\bigl(mp^n+(-1)^nmp^n+1-\alpha^n-(-\alpha)^n-\beta^n-(-\beta)^n\bigr)\frac{T^n}{n}\right)\\
&=\frac{(1-\alpha T)(1+\alpha T)(1-\beta T)(1+\beta T)}{(1-T)(1-pT)^m(1+pT)^m}\\
%=\frac{(1-\alpha^2T^2)(1-\beta^2T^2)}{(1-T)(1-pT)^{(\frac{2+r}{2})}(1+pT)^m} 
&= \frac{L_p(E_1/\F_{p^2},T^2)}{(1-T)(1-pT)^m(1+pT)^m},
\end{align*}
and the result follows from Theorem \ref{th:ZetaEuler}.

(ii) If $r$ is odd, say $r=2k+1$, then the $2k$  $\mathbb{P}^1$s in the chain linking $E_1$ and $E_2$ are defined over~$\F_{p^2}$ and permuted pairwise, while the ``central'' $\mathbb{P}^1$ is defined over $\F_p$. 
 %so that $\Frob(E_1)=E2$  \textcolor{red}{Clarify}.
It follows that we have $P_2(T) = (1-pT)^{\frac{2+2k}{2}}(1+pT)^{\frac{2+2k}{2}}(1-pT)$.

As in (i), let $\alpha, \beta$ be algebraic numbers such that $L_p(E_1,T) =(1-\alpha T)(1-\beta T)$.
% In particular, 
%for $n \in \Z_{>0}$ even, 
%$|E_i(\F_{p^n})| = p^n+1-\alpha^n-\beta^n$ for $i=1,2$. 
It follows that
\[
|\bar{C}(\F_{p^n})| = 
\begin{cases}
p^n +1, \text{ if } n \text{ is odd,}\\
(2k+2)p^n+p^n-2\alpha^n-2\beta^n-(2k+2), \text{ if } n \text{ is even.}\\
\end{cases}
\]
Therefore
\begin{align*}
Z(\bar{C}/\F_p,T) &=
%= \exp(\sum_{n\ge1}|\bar{C}(F_{p^n})|\frac{T^n}{n})
\exp\left(\sum_{n\ge1}((k+1)p^n+(-1)^n(k+1)p^n+p^n+1-\alpha^n-(-\alpha)^n-\beta^n-(-\beta)^n)\frac{T^n}{n}\right)\\
&=\frac{(1-\alpha T)(1+\alpha T)(1-\beta T)(1+\beta T)}{(1-T)(1-pT)^{k+1}(1+pT)^{k+1}(1-pT)} \\
%=\frac{(1-\alpha^2T^2)(1-\beta^2T^2)}{(1-T)(1-pT)^{k+1}(1+pT)^{k+1}(1-pT)} 
&= \frac{L_p(E_1/\F_{p^2},T^2)}{(1-T)(1-pT)^{k+1}(1+pT)^{k+1}(1-pT)},
\end{align*}
and the result follows from Theorem \ref{th:ZetaEuler}.
\end{proof}

\begin{corollary}\label{co:E1E2a}
Let $C\colon y^2=f(x)$ and $p$ be as in Proposition~\ref{pr:split} with type \textbf{1}, \textbf{2a}, \textbf{4}, and let $\tilde f = p^{-v_p(c)}f\in \Z[x]$.  Let $L$ be the splitting field of $f$ over $\Q_p$, let $r_1\in \OL$ be a root in the cluster of depth $n$, and for types \textbf{2a}, \textbf{4} let $r_2\in \OL$ be a root in the cluster of depth $m$.  Fix $s_1\in \Z$ with $r_1\equiv s_1\bmod p^n\OL$, and for types \textbf{2a}, \textbf{4}, $s_2\in \Z$ with $r_2\equiv s_2\bmod p^m\OL$.
The elliptic curves $E_1/\F_p$ and $E_2/\F_p$ of Proposition~\ref{pr:Lp} may be explicitly computed as follows:
\begin{itemize}
\setlength{\itemsep}{4pt}
\item[\textbf{1})]  $E_1\colon y^2 = \bar g_1(x)$, where $\bar g_1\in \F_p[x]$ is the squarefree part of $\tilde f\bmod p$, and\\
                    $E_2\colon y^2 = \bar g_2(x)$, where $\bar g_2(x) = \tilde f(p^{n}x+s_1)/p^{3n}\bmod p\in \F_p[x]$.
\item[\textbf{2a})] $E_1\colon y^2 = \bar g_1(x)$, where $\bar g_1(x) = \tilde f(p^{n}x+s_1)/p^{3n}\bmod p\in \F_p[x]$, and\\
                    $E_2\colon y^2 = \bar g_2(x)$, where $\bar g_2(x) = \tilde f(p^{m}x+s_2)/p^{3m}\bmod p\in \F_p[x]$.
\item[\textbf{4})]  $E_1\colon y^2 = \bar g_1(x)$, where $\bar g_1\in \F_p[x]$ is the squarefree part of $\tilde f(p^nx+s_1)/p^{5n}\bmod p$, and\\
				     $E_2\colon y^2 = \bar g_2(x)$, where $\bar g_2(x) = \tilde f(p^{m}x+s_2)/p^{3m+2n}\bmod p$.
\end{itemize}
\end{corollary}
\begin{proof}
This follows from Definition 8.5 and Theorem 8.6 in \cite{m2d2}, where explicit models for the components associated to principal clusters are constructed. We explicitly work out the case of Type~\textbf{1}, the other types are similar.  
There are two principal clusters in the cluster picture of Type \textbf{1}: the full set of roots $\cR=\{\alpha_1,\alpha_2,\alpha_3,\beta_1, \beta_2, \beta_3\}$ ordered from left to right in the cluster picture, and the cluster of depth $n$, say $\s$.
We now follow Definition 8.5 in \cite{m2d2}, which associates a component to each principal cluster. Recall that since $C$ is semistable, $v_p(c) \in 2\Z$ (Definition 1.8 and Theorem 1.9 in \cite{m2d2}), and therefore $v_p(c)=0$ by our assumption. For a root $r$, let $\bar{r}$ denote $r \bmod p\OL$.
We have $\Gamma_{\cR} : Y^2=\bar{c} (X-\bar{\alpha_1}) (X-\bar{\alpha_2}) (X-\bar{\alpha_3}) (X-\bar{\beta_1})$. This gives $E_1/\F_p$ since $f \equiv \bar{c} (X-\bar{\alpha_1}) (X-\bar{\alpha_2}) (X-\bar{\alpha_3}) (X-\bar{\beta_1})^3 \bmod p$. Note that $\bar{\beta}_1 \in \F_p$ but that $\bar{\alpha}_1, \bar{\alpha}_2, \bar{\alpha}_3$ may be permuted by Frobenius. 

For $E_2$, we look at the component associated to $\s$. Choose $z_\s = r_1$. We have $v(r_1-\alpha_i)=0$ for $i=1,2,3$, thus $c_\s= \bar{c}(s_1-\bar{\alpha_1})(s_1-\bar{\alpha_2})(s_1-\bar{\alpha_3})$. Therefore
\[
\Gamma_\s : Y^2 = \bar{c}\,\bigl(s_1-\bar\alpha_1\bigr)\bigl(s_1-\bar\alpha_2\bigr)\bigl(s_1-\bar\alpha_3\bigr)\left(X-\overline{\frac{\beta_1-r_1}{p^n\OL}}\right)\left(X-\overline{\frac{\beta_2-r_1}{p^n\OL}}\right)\left(X-\overline{\frac{\beta_3-r_1}{p^n\OL}}\right),
\]
which is $f(p^nx + s_1)/p^{3n} \bmod p$ as claimed.  
\end{proof}

\begin{corollary}\label{co:E1E2b}
Let $C\colon y^2=f(x)$ and $p$ be as in Proposition~\ref{pr:split} with type  \textbf{2b}, let $\tilde f = p^{-v_p(c)}f$, and let $L$ be the splitting field of $f$ over $\Q_p$.
Let $u\in \Z[x]$ be any lift of the irreducible quadratic $\bar u \in \F_p[x]$ in case \textbf{2b} of Proposition~\ref{pr:split}, and define $F\coloneqq \Q_p[z]/(u(z))\subseteq L$, $\mathcal O\coloneqq \Z[z]/(u(z))\subseteq \OL$, and $\kappa \coloneqq \F_p[z]/(\bar u(z))\simeq \F_{p^2}$. Let $r\in \OL$ be a root of $f$, and let $s\in \mathcal O$ satisfy $r\equiv s\bmod p\OL$.
Let $\hat f$ denote the image of $f$ in $\mathcal O[x]$ via the inclusion $\Z[z]\subseteq \mathcal O[z]$ induced by $\Z\subseteq \mathcal O$,
and let $\bar g = \hat f(p^nx+s)/p^{3n}\bmod p\mathcal O\in \kappa[x]\simeq \F_{p^2}[x]$.
Then $y^2=\bar g(x)$ and its $\Gal(\kappa/\F_p)$-conjugate are models for the elliptic curves $E_1/\F_{p^2}$ and $E_2/\F_{p^2}$ of Proposition~\ref{pr:Lp}.
\end{corollary}
\begin{proof}
Denote $\s_1$ and $\s_2$ the two clusters of size 3 in the  picture. As for Type \textbf{1} in the proof above, we closely follow Definition 8.5 and Theorem 8.6 in \cite{m2d2}. Both $\s_1$ and $\s_2$ are principal. Since the Frobenius automorphism permutes both clusters, the associated components ($E_1$ and $E_2$) must be Galois conjugate as mentioned in Remark \ref{rk:GaloisConj}. We construct $E_1$ explicitly. Denote $\alpha_1,\alpha_2,\alpha_3$ the roots in $\s_1$ and $\beta_1, \beta_2, \beta_3$ the roots in $\s_2$. Permuting indices if necessary, Frobenius permutes $\alpha_i$ and $\beta_i$ for $i=1,2,3$. 
We choose $r=z_{\s_1}$ a center for $\s_1$. In particular, $r$ is one of the $\alpha$s. Since $v(r-\beta_i) = 0$ for $i=1,2,3$, we have $c_{\s_1} = \bar{c}(s-\bar{\beta_1})(s-\bar{\beta_2})(s-\bar{\beta_3})$, where for a root $\beta$, we have $\bar{\beta} = \beta \bmod p\OL$. It follows that $E_1/\F_{p^2}$ is given by 
\[
\Gamma_{\s_1} : Y^2 = c_{\s_1}\left(X-\overline{\frac{r-\alpha_1}{p^n\OL}}\right)\left(X-\overline{\frac{r-\alpha_2}{p^n\OL}}\right)\left(X-\overline{\frac{r-\alpha_3}{p^n \OL}}\right),
\]
which is $\hat f(p^nx+s)/p^{3n}\bmod p\mathcal O$ since $v(r-\alpha_i) = n$ for $i=1,2,3$.
Now the construction of $\Gamma_{\s_2}$, which defines $E_2/\F_{p^2}$, is the same as that of $\Gamma_{\s_1}$ with the roots $\alpha$s and $\beta$s swapped. Since Frobenius permute them pairwise, it follows that $E_2$ is given by the $\Gal(\kappa/\F_p)$-conjugate of $\bar{g(x)}$. 
\end{proof}

%%%%%%%%%%%%%%%%%%%%%%%%%%%%%%%%%
%%%%%%%%%%%%%%%%%%%%%%%%%%%%%%%%%
\section{Algorithms}\label{sec:alg}
%%%%%%%%%%%%%%%%%%%%%%%%%%%%%%%%%
%%%%%%%%%%%%%%%%%%%%%%%%%%%%%%%%%

In this section we describe our algorithm for computing $L_p(C,T)$ for a genus 2 curve $C/\Q$ at an odd prime $p$ of almost good reduction; so $p$ divides the minimal discriminant $\Delta(C)\in\Z$ of $C$ but does not divide the conductor $N(C)\in\Z$ of its Jacobian. Computing the set of primes that satisfy this assumption is at least as hard as factoring $\Delta(C)$, a problem for which no polynomial-time algorithm is known, but there are efficient algorithms to compute $\Delta(C)$ \cite{Liu24} and the $p$-adic valuation of $N(C)$ at a given odd prime \cite{Liu94}, and these have been widely implemented in computer algebra systems such as \textsc{Magma}~\cite{magma}, \textsc{Pari/GP}~\cite{pari}, and \textsc{SageMath}~\cite{sage}.
We shall henceforth assume that any prime $p$ provided as an input to our algorithms is a prime of almost good reduction for $C$.

As we will be working exclusively in rings of odd or zero characteristic, we may assume that $C$ is specified by an integral model of the form $y^2=f(x)$, where $f\in \Z[x]$ is a squarefree polynomial of degree 5 or 6 whose discriminant $\Delta(f)$ is divisible by $p$.
We will state complexity bounds for our algorithms in terms of the \defi{logarithmic height}
\[
\|f\|\coloneqq\log \max_i\{|f_i|\}
\]
of the input polynomial $f\in \Z[x]$, and the logarithm of the prime $p$.
The discriminant $\Delta(f)$ can be expressed as a homogeneous polynomial in the coefficients $f_i$ of degree 10.
It follows that $\log |\Delta(f)|=O(\|f\|)$ and $\log p=O(\|f\|)$.

In our complexity analyses we will always count bit operations.  For ease of reference we list asymptotic bit-complexity bounds for various operations used by our algorithms in Table~\ref{tab:bounds}, in which the parameter $b$ bounds the bit-sizes of the inputs.  We assume throughout that elements of the finite field $\F_p$ are uniquely represented as integers in $[0,p-1]$, so that reduction from $\Z$ to $\F_p$ amounts to computing a remainder modulo $p$.   We assume that elements of $\F_{p^2}\simeq \F_p[z]/(\bar g(z))$ are explicitly represented as linear polynomials $z$ that have been reduced modulo an irreducible quadratic polynomial $\bar g\in \F_p[x]$ that will be chosen by our algorithms.
The bounds in Table~\ref{tab:bounds} apply to operations on elements of $\F_p$, $\F_{p^2}$, $\Z$, and also to polynomials over these rings that have bounded degree (we will only consider polynomials of degree at most 6), as well as elements and polynomials of bounded degree over the ring $\mathcal O\coloneqq \Z[z]/(g(z))$, with $g\in \Z[x]$ an irreducible monic quadratic that we will use to represent elements of the monogenic order  $\mathcal O$ in the quadratic field $\Q[z]/(g(z))$.

\begin{table}
\begin{tabular}{lll}
operation & complexity & algorithm/reference\\\toprule
addition/subtraction & $O(b)$ & schoolbook algorithm \cite{GG}\\
multiplication & $O(b\log b)$ & fast integer multiplication \cite{HvdH}\\
reduction modulo $p$ & $O(b\log b)$ & fast Euclidean division \cite{GG}\\
greatest common divisor & $O(b\log^2\!b)$ & fast GCD \cite{GG}\\
Legendre symbol $(\frac{\cdot}{p})$ & $O(b\log^2\!b)$ & fast binary GCD \cite{BZ}\\
inversion in $\F_p^\times$ & $O(b\log^2\! b)$ & fast extended GCD \cite{GG}\\
square roots in $\F_p^\times$ given $s\not\in \F_p^{\times 2}$ & $O(b^2\log^2\!b/\log\log b)$ & fast Tonelli-Shanks \cite{S}\\
computing $L_p(E,T)$ for $E/\F_p$ or $E/\F_{p^2}$ & $O(b^5)$ & Schoof's algorithm \cite{Schoof,SS}\\\bottomrule
\end{tabular}
\bigskip

\caption{
Asymptotic complexity bounds for arithmetic operations performed by our algorithms.  Here $b$ denotes the number of bits used to represent the inputs, all of which we represent using $O(1)$ integers (as $\|f\|$ tends to infinity).  We use the $O(b\log b)$ bound on integer multiplication from \cite{HvdH} throughout.
}\label{tab:bounds}
\end{table}

In the descriptions of our algorithms that follow we will frequently need to reduce elements of characteristic zero rings modulo $p$, and also to lift elements of characteristic $p$ rings to characteristic zero.  In our implementations all elements of characteristic $p$ rings are represented as integers in $[0,p-1]$ or lists of such integers, so there is no actual computation involved in lifting, but when describing our algorithms we will use the notation
\[
a=\lift(\bar a)
\]
to indicate that $a$ is the lift of $\bar a$ from a characteristic $p$ ring (such as $\F_p$, $\F_p[x]$, $\F_{p^2}\simeq\F_p[z]/(\bar g(z))$, $\F_{p^2}[x]$) to the corresponding ring of characteristic zero ($\Z$, $\Z[x]$, $\mathcal O=\Z[z]/(g(z))$, $\mathcal O[x]$).  For the sake of clarity we use the overline notation ``$\bar a$'' to indicate that $\bar a$ is an element of a characteristic $p$ ring, and the absence of an overline in the notation ``$a$'' indicates that $a$ lives in characteristic zero.

As can be seen from the cluster pictures in Proposition~\ref{pr:E1E2}, while the polynomial $f\in \Z[x]$ that defines the curve $C\colon y^2=f(x)$ must be squarefree, the reduction of $p^{-v_p(c)}f$ modulo $p$ will typically have repeated factors (with repeated roots defined over $\F_p$ or $\F_{p^2}$); these factors are made explicit in Proposition~\ref{pr:split} in the case that $f$ is $p$-normalized.  Efficiently determining these repeated factors and their multiplicity plays a key role in our algorithms and motivates the following definition.

\begin{definition}
For each positive integer $k$ and polynomial $f\in \F_p[x]$ we define
\[
{\color{mydeficolor}\gcdk}(f)\coloneqq  \prod_{g^k|f} g^{v_g(f)-k+1} \in \F_p[x],
\]
where $g$ ranges over monic irreducible elements of $\F_p[x]$ and $v_g(f)=\max\{e\in \Z:g^e|f\}$.
\end{definition}

When $p>\deg(f)$ we can efficiently compute $\gcdk(f)$ via
\[
\gcdk(f) = \gcd\left(f,f^{(1)},\ldots,f^{(k-1)}\right),
\]
where the gcd on the right is understood to be monic. Assuming $\deg f=O(1)$ (we will always have $\deg f = 6$) this computation takes $O(b\log^2 b)$ time (we need $O(1)$ operations in $\F_p$, including an inversion needed to make the result monic).  For $p\le \deg f =O(1)$ we can compute $\gcdk(f)$ in $O(1)$ time by exhaustively testing $g^k|f$ for all monic $g$ of degree at most $\lfloor \deg f/k\rfloor$.

In order to apply the main results of the previous section, which assume we are in the setting of Corollary~\ref{co:types}, we may need to adjust the model of $C\colon y^2=f(x)$ to ensure that it is defined by a $p$-normalized polynomial $f$ (see Definition~\ref{def:pnorm}).  This leads to our first algorithm.

\begin{algorithm}\label{alg:normalize}
Given a squarefree $f(x)=\sum_i f_ix^i\in \Z[x]$ of degree $5$ or $6$ and an odd prime $p$, construct a polynomial $g(x)=\sum_i g_ix^i\in \Z[x]$ with $v_p(g_6)=\min_i\{v_p(g_i)\}\le 1$ for which the genus $2$ curve $y^2=g(x)$ has an outer cluster of depth zero at $p$ and is $\Q$-isomorphic to  $y^2=f(x)$.
\end{algorithm}

\noindent
\begin{enumerate}[\hspace{1em}1.]
\item If $\deg f = 5$, compute $f(a)$ for $a=0,1,2, \ldots$ until $f(a)\ne 0$, replace $f(x)$ by $f(x+a)$, and then replace $f(x)$ by $x^6f(1/x)$ so that $\deg f = 6$.
\item Let $v=v_p(f_6)$, and if $v>1$ or $v\ne \min_i\{v_p(f_i)\}$ then do the following:
\begin{enumerate}[\hspace{1em}a.]
\item Let $e=\max \{\lceil \frac{v-v_p(f_i))}{6-i}\rceil:0\le i \le 5\}$.
\item Replace $f(x)$ by $p^{6e-w}f(x/p^e)\in \Z[x]$, where $w=2\lfloor v/2\rfloor$, and replace $v$ by $v_p(f_6)$.
\end{enumerate}
\item Let $h=p^{-v}f$.
\item Repeat the following steps:
\begin{enumerate}[\hspace{1em}a.]
\item Let $\bar u = \gcd_6(h\bmod p)$ and if $\deg \bar u=0$ then go to step 5.
\item Replace $h(x)$ by $p^{-6}h(px+\lift(\bar a))$ where $\bar u=x-\bar a$.
\end{enumerate}
\item Return $g=p^vh$.
\end{enumerate}
\smallskip

\begin{proposition}\label{pr:normalization}
Algorithm~\ref{alg:normalize} is correct and runs in time $O(\|f\|^2\log^2\|f\|/\log p)$.
\end{proposition}
\begin{proof}
That $y^2=f(x)$ and $y^2=g(x)$ are $\Q$-isomorphic follows from the fact that the sextic forms $z^6f(x/z)$ and $z^6g(x/z)$ are related by an invertible linear transformation of $\mathbb P^1$ and multiplication by an even power of $p$, neither of which changes the $\Q$-isomorphic class of the corresponding genus 2 curves.
After step 2 we have $v_p(f_6)=\min_i\{v_p(f_i)\}\le 1$, since step 2b ensures $v_p(f_6)=v-w\in\{0,1\}$ and $v_p(g_i) = 6e-w-ie+v_p(f_i)\ge v-w =v_p(g_6)$ for $0\le i \le 5$.  Step 2 can increase the depth of the outer cluster by at most $e = O(\|f\|/\log p)$.

In step 4 we have $h(x)=\sum_i h_ix^i\in \Z[x]$ with $v_p(h_6)=0$.  The outer cluster of $y^2=h(x)$ at~$p$ will have nonzero depth  if and only if the roots of $h\bmod p$ all coincide, equivalently, if and only if $\bar u$ has positive degree.  If this happens then $\deg \bar u=1$ (since $\deg (h \bmod p) =6$), and $\bar u=x-\bar a$, where $\bar a$ has multiplicity 6 as a root of $h\bmod p$, and then replacing $h(x)$ by $p^{-6}h(px+\lift(\bar a))\in \Z[x]$ reduces the depth of the outer cluster by one.  Step 4 will terminate after $O(\|f\|/\log p)$ iterations when the depth reaches zero, and the genus 2 curve $y^2=g(x)=p^vh(x)$ has its outer cluster of depth zero.

For the time bound, let $b=\|f\|$.  We have $a\le 6$ in step 1, which performs $O(1)$ ring operations in~$\Z$, taking time $O(b\log b)$, yielding a new $f\in \Z[x]$ with $\|f\|=O(b)$. We can compute $v_p(f_i)$ in time $O(b\log b)$, which bounds the cost of steps 2, 3, and 5.  Each iteration of step 4 takes $O(b\log^2 b)$ time, and there are $O(\|f\|/\log p)$ iterations, yielding a total running time of $O(\|f\|^2\log^2\|f\|/\log p)$.
\end{proof}

\begin{remark}
For $p>5$ steps 1 and 2 of Algorithm~\ref{alg:normalize} can be replaced by the following: let $v=\min_i\{v_p(f_i)\}$, test $a=0,1,2,\ldots$ until $f(a)\not\equiv 0\bmod p^{v+1}$, replace $f(x)$ by $f(x+a)$, then replace $f(x)$ by $p^{-2\lfloor v/2\rfloor}x^6f(1/x)$. This has the virtue of not increasing the depth of the outer cluster, which potentially saves time in step 4, but this may not work when $p=3,5$.
\end{remark}

\begin{remark}
The loop in step 4 of Algorithm~\ref{alg:normalize} which is used to decrease the depth of the outer cluster can be viewed as computing a common initial $p$-adic approximation to the roots of $h(x)$, which all coincide modulo a power of $p$ equal to the depth of the outer cluster.  The same technique will be used in our algorithms below which work with the inner clusters.  If $h\bmod p$ had a simple root modulo $p$, or more generally, if we had an integer $s$ for which $v_p(h(s))>2v_p(h'(s))$, we could use Hensel lifting to approximate a root to $O(\|h\|/\log p)$ digits of $p$-adic precision in quasi-linear time (as a function of $\|h\|$), rather than the quasi-quadratic time required by step 4, since we can double the $p$-adic precision of our approximation in each step, rather than simply incrementing it.  But we are in precisely the situation where Hensel's lemma does not apply, and in general $h(x)$ need not have any $\Q_p$-rational roots, so there is no reason to expect that we can use Hensel lifting.
\end{remark}

\begin{remark}
Unlike all our remaining algorithms, Algorithm~\ref{alg:normalize} makes no assumptions about $f(x)$ other than requiring it to be squarefree of degree 5 or 6, so that $C\colon y^2=f(x)$ is a genus 2 curve.  Its output is $p$-normalized when $C$ has almost good reduction at $p$, which will be true in the context of our main algorithm (see Algorithm~\ref{alg:main} below) where it is used.
\end{remark}

If $f\in \Z[x]$ is a $p$-normalized polynomial we shall refer to the \defi{type} of $f$ as the type of the cluster picture of $y^2=f(x)$ at $p$, one of \textbf{1}, \textbf{2a}, \textbf{2b}, \textbf{4}.  Our next algorithm uses Proposition~\ref{pr:split} to efficiently determine the type of $f$.  Note that $\log p = O(\|f\|)$, since $p|\Delta(f)$ and $\log |\Delta(f)|=O(\|f\|)$.

\begin{algorithm}[\textsc{WhichType}]\label{alg:whichtype}
Given a $p$-normalized $f=\sum_if_ix^i\in \Z[x]$, determine its type.
\end{algorithm}
\noindent
\begin{enumerate}[\hspace{1em}1.]
\item Compute $\tilde f = p^{-v_p(f_6)}f\in \Z[x]$ and $\bar f = \tilde f\bmod p \in \F_p[x]$.
\item Compute $\bar g = \gcd_3(\bar f)$.
\item If $\deg \bar g=1$ then return \textbf{1}, and if $\deg \bar g = 3$ then return \textbf{4}.
\item If $\left(\frac{\Delta(\bar g)}{p}\right)=+1$ then return \textbf{2a}, otherwise return \textbf{2b}.
\end{enumerate}

\begin{proposition}\label{pr:whichtype}
Algorithm~\ref{alg:whichtype} is correct and runs in time $O(\|f\|\log^2\|f\|)$.
\end{proposition}
\begin{proof}
By Proposition~\ref{pr:split}, we have a unique triple root if we are in case \textbf{1}, two triple roots if we are in case \textbf{2a}, the cube of an irreducible quadratic if we are in case \textbf{2b}, and a unique quintuple root if we are in case \textbf{4}, and these are mutually exclusive.
The degree of $\bar g$ is thus 1, 2, or 3, depending on whether we are in case \textbf{1}, \textbf{2a}/\textbf{2b}, or \textbf{4}, respectively, and the cases \textbf{2a} and \textbf{2b} are distinguished by whether $\bar g$ has $\F_p$-rationals roots or not, equivalently, whether $\Delta(\bar g)$ is a square or not.

Let $b=\|f\|$.  From Table 1 we see that Step 1 takes $O(b\log b)$ time, the cost of step 3 is negligible, and steps 2 and 4 both take $O(b\log^2 b)$ time, which bounds the total complexity.
\end{proof}

Having determined the type of a $p$-normalized $f\in \Z[x]$, we can use Corollaries~\ref{co:E1E2a} and~\ref{co:E1E2b} to compute the $L$-polynomial of $C\colon y^2=f(x)$ at $p$.  To simplify the presentation we treat each of the four types separately.

\begin{algorithm}[\textsc{Type \textbf{1}}]\label{alg:type1}
Given a $p$-normalized polynomial $f\in\Z[x]$ of type \textbf{1}, compute the $L$-polynomial of $C\colon y^2=f(x)$ at $p$.
\end{algorithm}
\noindent
\begin{enumerate}[\hspace{1em}1.]
\item Compute $\bar f \equiv f\bmod p \in \F_p[x]$ and $\gcd_3(\bar f)=x-\bar r\in \F_p[x]$ (per Proposition~\ref{pr:split}).
\item Let $\bar g_1=\bar f(x+\bar r)/x^2\in \F_p[x]$ (a quartic) and let $E_1\colon y^2=\bar g_1(x)$.
\item Let $r = \lift(\bar r)$ and repeat the following steps:
\begin{enumerate}[\hspace{1em}a.]
\item Replace $f$ with $f(px+r)/p^3\in \Z[x]$ and let $\bar g_2 \equiv f\bmod p\in \F_p[x]$ (a cubic).
\item If $\Delta(\bar g_2)\ne 0$, then let $E_2\colon y^2=\bar g_2(x)$ and go to step 4.
\item Compute $\gcd_3(\bar g_2)=x-\bar r\in \F_p[x]$ and replace $r$ with $\lift(\bar r)$.
\end{enumerate}
\item Return $L_p(C,T)=L_p(E_1,T)L_p(E_2,T)\in \Z[T]$.
\end{enumerate}

\begin{remark}
Step 3b only needs to be performed every second iteration, since the depth must be even, by Proposition~\ref{pr:E1E2}.
\end{remark}

\begin{proposition}\label{pr:type1}
Algorithm~\ref{alg:type1} is correct and runs in time $O\bigl(\|f\|^2\log^2\!\|f\|/\log p+\log^5\! p\bigr)$.
\end{proposition}
\begin{proof}
Step 2 computes $\bar h$ as the squarefree part of $\bar f$, while step 3 iteratively computes the polynomial $\bar g_2(x)=f(p^nx+s_1)/p^{3n}\bmod p$, where $s_1\equiv r_1\bmod p^n$ for some root $r_1$ in the inner cluster of depth $n$; correctness follows from Corollary~\ref{co:E1E2a}.

Let $b=\|f\|$.  From Table~\ref{tab:bounds} we see that step 1 takes $O(b\log^2\!b)$ time, step 2 takes $O(b\log b)$ time, each iteration of step 3 takes $O(b\log^2\!b)$ time, and step 4 takes $O(\log^5\! p)$ time. There are $n=O(\|f\|/\log p)$ iterations of step 3, and this yields the desired complexity bound.
\end{proof}

Our algorithm for type \textbf{2a} is the one case where it is difficult to give an efficient deterministic algorithm because we need to compute the roots of a quadratic polynomial over $\F_p$.  This can be done efficiently using a probabilistic algorithm, or by assuming the extended Riemann hypothesis, but in order to give an unconditional deterministic algorithm we will assume we are given a nonsquare $s\in \F_p^\times$.  We can use $s$ to deterministically compute the square root of any element of $\F_p$ via the Tonelli-Shanks algorithm, which effectively computes square roots using a deterministic discrete logarithm computation in the 2-Sylow subgroup $H$ of $\F_p^\times$ with respect to a given generator of $H$, which we can take to be $s^m$ where $p=2^em+1$ with $m$ odd.

\begin{remark}\label{rem:rands}
A nonsquare $s\in\F_p^\times$ can be precomputed in $O(\log p(\log\log p)^2)$ expected time by picking random $s\in \F_p^\times$ until one finds $(\frac{s}{p})= -1$; this depends only on $p$, not $f$.
\end{remark}

\begin{algorithm}[\textsc{Type \textbf{2a}}]\label{alg:type2a}
Given a $p$-normalized polynomial $f=\sum_i f_ix^i\in\Z[x]$ of type \textbf{2a} and a nonsquare $s\in \F_p^\times$, compute the $L$-polynomial of $C\colon y^2=f(x)$ at $p$.
\end{algorithm}
\begin{enumerate}[\hspace{1em}1.]
\item Compute $\tilde f=p^{-v_p(f_6)}f\in \Z[x]$ and $\bar f \equiv \tilde f\bmod p \in \F_p[x]$.
\item Compute $\bar u = \gcd_3(\bar f)=(x-\bar r_1)(x-\bar r_2)\in \F_p[x]$ (per Proposition~\ref{pr:split}).
\item Compute the roots $\bar r_1,\bar r_2\in \F_p$ of $\bar u$ via the quadratic formula, using $s$ to compute $\sqrt{\Delta(\bar u)}$.
\item Use $\bar r_1$ and $\bar r_2$ to compute  elliptic curves $E_1$ and $E_2$ as in step 3 of Algorithm~\ref{alg:type1} (using $f=\tilde f$).
\item Return $L_p(C,T)=L_p(E_1,T)L_p(E_2,T)\in \Z[T]$.
\end{enumerate}

\begin{proposition}\label{pr:type2a}
Algorithm~\ref{alg:type2a} is correct and runs in time $O\bigl(\|f\|^2\log^2\|f\|/\log p+ \log^5\!p\bigr)$.
\end{proposition}
\begin{proof}
Correctness of the algorithm follows from Corollary~\ref{co:E1E2a}, and the complexity analysis is as in the proof of Proposition~\ref{pr:type1}, once we note that step 3 takes $O(b^2\log^2\!b/\log\log b)$ time, with $b=\log p$ in Table~\ref{tab:bounds}, which is bounded by the cost of step 5.
\end{proof}

For type \textbf{2b} we work in the setting of Corollary~\ref{co:E1E2b}. If we put $\tilde f\coloneqq p^{-v_p(f_6)}f\in \Z[x]$ then Proposition~\ref{pr:split} implies that $\tilde f\bmod p$ has the form $\bar c \cdot\bar u^3\in \F_p[x]$ with $\bar c\in \F_p^\times$ and $\bar u$ a monic quadratic.  We let $u=\lift(\bar u)\in \Z[x]$ and define $\mathcal O\coloneqq \Z[z]/(u(z))$ (an order in the quadratic field $\Q[z]/(u(z))$) and $\kappa\coloneqq \F_p[z]/(\bar u(z))\simeq \mathcal O/p\mathcal O\simeq\F_{p^2}$.  The reduction map $\pi\colon \mathcal O\to \mathcal O/p\mathcal O \overset{\sim}{\rightarrow} \kappa$ sends the image of $z$ in~$\mathcal O$ to the image of $z$ in $\kappa$ and integers to their reductions modulo $p$, while the map $\bar a\mapsto \lift(\bar a)$ is the unique section of $\pi$ that sends elements of $\F_p$ to integers in $[0,p-1]$. We use $\hat f$ to denote the image of $\tilde f$ under the embedding $\Z[x]\hookrightarrow \mathcal O[x]$ induced by $\Z\hookrightarrow \mathcal O$.

\begin{algorithm}[\textsc{Type \textbf{2b}}]\label{alg:type2b}
Given a $p$-normalized polynomial $f\in\Z[x]$ of type \textbf{2b}, compute the $L$-polynomial of $C\colon y^2=f(x)$ at $p$.
\end{algorithm}
\begin{enumerate}[\hspace{1em}1.]
\item Compute $\tilde f=p^{-v_p(f_6)}f\in \Z[x]$ and $\bar f \equiv \tilde f\bmod p \in \F_p[x]$.
\item Compute $\bar u = \gcd_3(\bar f)= x^2+\bar u_1 x+\bar u_2\in \F_p[x]$ (per Proposition~\ref{pr:split}).
\item Let $\mathcal O\coloneqq \Z[z]/(u(z))$ and $\kappa \coloneqq \F_p[z]/(\bar u(z))$, where $u = \lift(\bar u)$, and $\pi\colon \mathcal O\to \kappa$ be as above.
\item Let $\hat f$ be the image of $\tilde f$ in $\mathcal O[x]$, let $r=z\in \mathcal O$, and repeat the following steps:
\begin{enumerate}[\hspace{1em}a.]
\item Replace $\hat f$ with $\hat f(px+r)/p^3\in \mathcal O [x]$ and let $\bar g =\pi(\hat f)\in \kappa[x]\simeq \F_{p^2}[x]$ (a cubic).
\item If $\Delta(\bar g)\ne 0$, then let $E\colon y^2=\bar g(x)$ and go to step 5.
\item Compute $\gcd_3(\bar g)=x-\bar r\in \kappa[x]\simeq \F_{p^2}[x]$ and replace $r$ with $\lift(\bar r)$.
\end{enumerate}
\item Return $L_p(C,T)=L_p(E,T^2)\in \Z[T]$.
\end{enumerate}

\begin{proposition}\label{pr:type2b}
Algorithm~\ref{alg:type2a} is correct and runs in time $O\bigl(\|f\|^2\log^2\|f\|/\log p+ \log^5\!p\bigr)$.
\end{proposition}
\begin{proof}
Correctness follows from Corollary~\ref{co:E1E2b}, and the complexity analysis is as in the proof of Proposition~\ref{pr:type1}; the fact that we are working in $\mathcal O$ rather than $\Z$ and $\kappa\simeq \F_{p^2}$ rather than $\F_p$ changes the constant factors, but the asymptotic complexity bound is the same.
\end{proof}

\begin{algorithm}[\textsc{Type \textbf{4}}]\label{alg:type4}
Given a $p$-normalized polynomial $f\in\Z[x]$ of type \textbf{4}, compute the $L$-polynomial of $C\colon y^2=f(x)$ at $p$.
\end{algorithm}
\noindent
\begin{enumerate}[\hspace{1em}1.]
\item Compute $\tilde f=p^{-v_p(f_6)}f\in \Z[x]$ and $\bar f \equiv \tilde f\bmod p \in \F_p[x]$.
\item Compute $\gcd_5(\bar f)= x-\bar r\in \F_p[x]$ (per Proposition~\ref{pr:split}).
\item Let $r=\lift(\bar r)$ and repeat the following steps:
\begin{enumerate}[\hspace{1em}a.]
\item Replace $\tilde f$ with $\tilde f(px+r)/p^5\in \Z[x]$ and let $\bar f \equiv \tilde f\bmod p\in \F_p[x]$ (a quintic).
\item If $\gcd_3(\bar f)=(x-\bar s)^e$ has degree $e=1$ then let $E_1\colon y^2=\bar f(x)/(x-\bar s)^2$ and go to step 4.
\item Compute $\gcd_5(\bar f)=x-\bar r\in \F_p[x]$ (per Proposition~\ref{pr:split}).
\end{enumerate}
\item Let $r=\lift (\bar s)$ and repeat the following steps:
\begin{enumerate}[\hspace{1em}a.]
\item Replace $\tilde f$ with $\tilde f(px+r)/p^3\in \Z[x]$ and let $\bar g \equiv \tilde f\bmod p\in \F_p[x]$ (a cubic).
\item If $\Delta(\bar g)\ne 0$, then let $E_2\colon y^2=\bar g(x)$ and go to step 5.
\item Compute $\gcd_3(\bar f) = x-\bar r$ and replace $r$ by $\lift(\bar r)$.
\end{enumerate}
\item Compute $L_p(C,T)=L_p(E_1,T)L_p(E_2,T)$.
\end{enumerate}

\begin{proposition}\label{pr:type4}
Algorithm~\ref{alg:type4} is correct and runs in time $O\bigl(\|f\|^2\log^2\|f\|/\log p+ \log^5\!p\bigr)$.
\end{proposition}
\begin{proof}
Correctness follows from Corollary~\ref{co:E1E2a}, and the complexity analysis is as in the proof of Proposition~\ref{pr:type1}; steps 3 and 4 of Algorithm~\ref{alg:type4} have the same complexity as step 3 of Algorithm~\ref{alg:type1}.
\end{proof}

We now present our main algorithm
\begin{algorithm}[\textsc{Main}]\label{alg:main}
Given a squarefree polynomial $f\in \Z[x]$ of degree 5 or 6 defining a genus $2$ curve $C\colon y^2=f(x)$ with almost good reduction at an odd prime $p$, and a nonsquare $s\in \F_p^\times$, compute the $L$-polynomial of $C$ at $p$.
\end{algorithm}
\noindent
\begin{enumerate}[\hspace{1em}1.]
\item Use Algorithm~\ref{alg:normalize} to replace $f$ with a $p$-normalized polynomial $f$.
\item Use Algorithm~\ref{alg:whichtype} to determine the type of $f$ (\textbf{1}, \textbf{2a}, \textbf{2b}, or \textbf{4}).
\item Use whichever of Algorithms~\ref{alg:type1}-\ref{alg:type4} matches the type to compute $L_p(C,T)\in \Z[T]$.
\end{enumerate}

\begin{proposition}\label{pr:main}
Algorithm~\ref{alg:main} is correct and runs in time $O(\|f\|^2\!\log^2\!\|f\|/\!\log p + \log^5\! p)$.
\end{proposition}
\begin{proof}
This follows immediately from Propositions~\ref{pr:normalization}, \ref{pr:whichtype}, \ref{pr:type1}, \ref{pr:type2a}, \ref{pr:type2b}, \ref{pr:type4}.
\end{proof}

\begin{corollary}
There is a probabilistic implementation of Algorithm~\ref{alg:main} that does not require the input $s\in \F_p^\times$ and runs in $O(\|f\|^2\!\log^2\!\|f\|/\!\log p + \log^5\! p) \subseteq O(\|f\|^5)$ expected time.
\end{corollary}
\begin{proof}
As noted in Remark~\ref{rem:rands}, we can compute a nonsquare $s\in \F_p^\times$ in $O(\log p (\log\log p)^2)$ expected time and then apply Algorithm~\ref{alg:main}.
\end{proof}

\section{Implementation}\label{sec:implementation}

A simple Magma implementation of Algorithms~\ref{alg:normalize}--\ref{alg:main} is available in the \href{https://github.com/AndrewVSutherland/Genus2Euler}{\textsc{Genus2Euler}} GitHub repository associated to this paper \cite{repo}.  There is also a low-level C implementation based on the \textsc{smalljac} library \cite{KS08} that is still being refined; we report preliminary timings here.

We tested our algorithms on a dataset of approximately 2.5 million genus 2 curves $C/\Q$ of small conductor that have almost good reduction at some prime $p$.  This dataset covers a conductor range comparable to the current database of genus 2 curves in the LMFDB ($2^{20}$ versus $10^6$), but spans a much larger discriminant range (the largest minimal discriminant is $|\Delta(C)|\approx 10^{217}$ versus $|\Delta(C)|\le 10^6$ in the LMFDB).

Many of these curves have almost good reduction at more than one odd prime, and the total number of $(C,p)$ pairs is $3\,454\,506$.  We attempted to compute each of the corresponding Euler factors in three ways: (1) using the \href{http://magma.maths.usyd.edu.au/magma/handbook/text/1558#17913}{\textsc{EulerFactor}} intrinsic in Magma \cite{magma}, (2) using the Magma implementation of Algorithm~\ref{alg:main} available in \cite{repo}, (3) using a \textsc{smalljac}-based C implementation of Algorithm~\ref{alg:main}.  We ran our tests on a server equipped with dual 128-core AMD EPYC 9754 CPUs running at 2.25GHz and 1.5TB memory, with at most 256 tests running in parallel to avoid hyperthreading.  All reported times are CPU times for a single core.  We repeated each computation using our Magma and C implementations of Algorithm~\ref{alg:main} for 100 and 10\,000 iterations, respectively, to increase the accuracy of the timings (the algorithms are deterministic so there is very little variance in running times on the same input, but it is difficult to accurately time computations that take less than one millisecond), but we ran the computations using \textsc{EulerFactor} only once, due to the time involved; we terminated 489 of the tests involving \textsc{EulerFactor} that did not finish within 8 hours.

The total, average, median, and maximum running times to compute Euler factors for the $3\,454\,506$ pairs $(C,p)$ are shown below, with the 489 cases not completed by  \textsc{EulerFactor} capped at 8 hours.

\begin{table}[h]
\begin{tabular}{lllll}
method & total time & average time & median time & maximum time\\\toprule
\textsc{EulerFactor} & 242 days & 6.1 seconds &  0.9 seconds & over 8 hours\\
Algorithm \ref{alg:main} (Magma) & 1.23 hours & 1.3 milliseconds & 1.2 milliseconds & 24 milliseconds\\
Algorithm \ref{alg:main} (C) & 27.1 seconds & 7.8 microseconds & 3.4 microseconds & 21 milliseconds\\
\bottomrule
\end{tabular}
\medskip

\caption{Timings for computing $3\,454\,506$ Euler factors of genus 2 curves $C/\Q$ of small conductor at odd primes of almost good reduction.}\label{tab:summary}
\end{table}

The maximum time for the C implementation of Algorithm~\ref{alg:main} is larger than one might expect (relative to the Magma implementation of Algorithm~\ref{alg:main}) because in some of the \textbf{2b} cases we need to count points on elliptic curves over $\F_{p^2}$ with $p\approx 2^{30}$, a regime for which the \textsc{smalljac} library has not been optimized (it is much faster over prime fields); if one excludes the type \textbf{2b} cases the maximum running time for our C implementation drops to 34 microseconds, even with $p\approx 2^{35}$.

More detailed timing information broken out by type for small primes $p\le 2^{10}$ (about 90\% of the test cases) can be found in Table~\ref{tab:timings}.
Finally, Tables~\ref{tab:hard} and~\ref{tab:fail} list ten examples of pairs $(C,p)$ where Magma's \textsc{EulerFactor} function struggled; Table~\ref{tab:hard} lists 10 examples where \textsc{EulerFactor} took roughly an hour, and Table~\ref{tab:fail} lists 10 of the 489 examples where \textsc{EulerFactor} failed to terminate within 8 hours.  A complete list of these 489 cases can be found in \cite{repo}.

\begin{table}
\renewcommand{\arraystretch}{0.92}
\begin{small}
\begin{tabular}{ccrrrrrrr}
&&&\multicolumn{2}{c}{\textsc{EulerFactor}}&\multicolumn{2}{c}{Alg 7 \textsc{Magma}}&\multicolumn{2}{c}{Alg 7 \textsc{C}}\\
\cmidrule(lr){4-5}\cmidrule(lr){6-7}\cmidrule(lr){8-9}
prime(s) & type & count &$\mu$&$\sigma$&$\mu$&$\sigma$&$\mu$&$\sigma$\\
\toprule
$3$ & \textbf{1} & 9413 & 1140 & 270 & 1.290 & 0.079 & 0.001872 & 0.000238\\
& \textbf{2a} & 20863 & 538 & 519 & 1.284 & 0.072 & 0.002346 & 0.000601\\
& \textbf{2b} & 53815 & 1151 & 270 & 0.942 & 0.080 & 0.002352 & 0.000699\\
& \textbf{4} & 19699 & 355 & 429 & 1.212 & 0.052 & 0.002185 & 0.000244\\
\midrule
$5$ & \textbf{1} & 65163 & 742 & 541 & 1.461 & 0.053 & 0.001979 & 0.000134\\
& \textbf{2a} & 35374 & 460 & 499 & 1.492 & 0.052 & 0.002248 & 0.000522\\
& \textbf{2b} & 145388 & 1163 & 270 & 1.092 & 0.057 & 0.002098 & 0.000604\\
& \textbf{4} & 36715 & 314 & 400 & 1.349 & 0.054 & 0.002115 & 0.000133\\
\midrule
$7$ & \textbf{1} & 76975 & 572 & 539 & 1.181 & 0.060 & 0.002157 & 0.000092\\
& \textbf{2a} & 24187 & 361 & 444 & 1.196 & 0.061 & 0.002488 & 0.000468\\
& \textbf{2b} & 134170 & 1158 & 269 & 0.828 & 0.054 & 0.002829 & 0.000803\\
& \textbf{4} & 27771 & 249 & 331 & 1.147 & 0.059 & 0.002295 & 0.000109\\
\midrule
$2^3-2^4$ & \textbf{1} & 149872 & 412 & 470 & 1.283 & 0.064 & 0.002166 & 0.000102\\
& \textbf{2a} & 21027 & 293 & 380 & 1.317 & 0.076 & 0.002461 & 0.000395\\
& \textbf{2b} & 188755 & 1166 & 271 & 0.989 & 0.080 & 0.005278 & 0.002037\\
& \textbf{4} & 25475 & 204 & 258 & 1.192 & 0.076 & 0.002299 & 0.000078\\
\midrule
$2^4-2^5$ & \textbf{1} & 249467 & 351 & 443 & 1.228 & 0.057 & 0.002245 & 0.000168\\
& \textbf{2a} & 13877 & 265 & 359 & 1.254 & 0.056 & 0.002544 & 0.000374\\
& \textbf{2b} & 271493 & 1198 & 271 & 1.089 & 0.143 & 0.006651 & 0.002138\\
& \textbf{4} & 18862 & 196 & 245 & 1.108 & 0.080 & 0.002380 & 0.000121\\
\midrule
$2^5-2^6$ & \textbf{1} & 191028 & 313 & 407 & 1.252 & 0.053 & 0.002429 & 0.000178\\
& \textbf{2a} & 3085 & 226 & 307 & 1.284 & 0.047 & 0.002788 & 0.000359\\
& \textbf{2b} & 174822 & 1217 & 284 & 1.697 & 0.577 & 0.008004 & 0.002155\\
& \textbf{4} & 5141 & 185 & 220 & 1.098 & 0.069 & 0.002628 & 0.000074\\
\midrule
$2^6-2^7$ & \textbf{1} & 185005 & 290 & 390 & 1.287 & 0.051 & 0.002625 & 0.000178\\
& \textbf{2a} & 851 & 233 & 313 & 1.324 & 0.061 & 0.003039 & 0.000359\\
& \textbf{2b} & 161260 & 1379 & 602 & 1.189 & 0.867 & 0.009324 & 0.002153\\
& \textbf{4} & 2087 & 188 & 230 & 1.105 & 0.067 & 0.002868 & 0.000101\\
\midrule
$2^7-2^8$ & \textbf{1} & 151925 & 288 & 395 & 1.420 & 0.089 & 0.003010 & 0.000213\\
& \textbf{2a} & 191 & 256 & 356 & 1.485 & 0.107 & 0.003422 & 0.000406\\
& \textbf{2b} & 133090 & 1858 & 2779 & 1.190 & 0.677 & 0.010872 & 0.002523\\
& \textbf{4} & 684 & 229 & 304 & 1.169 & 0.103 & 0.003292 & 0.000193\\
\midrule
$2^8-2^9$ & \textbf{1} & 131149 & 286 & 402 & 1.253 & 0.053 & 0.003851 & 0.000357\\
& \textbf{2a} & 43 & 190 & 222 & 1.277 & 0.052 & 0.004163 & 0.000525\\
& \textbf{2b} & 108496 & 2623 & 7999 & 1.338 & 0.754 & 0.012742 & 0.002923\\
& \textbf{4} & 203 & 324 & 460 & 1.083 & 0.067 & 0.004166 & 0.000336\\
\midrule
$2^9-2^{10}$ & \textbf{1} & 98664 & 314 & 430 & 1.275 & 0.046 & 0.007900 & 0.001029\\
& \textbf{2a} & 3 & 170 & 16 & 1.267 & 0.047 & 0.008911 & 0.000183\\
& \textbf{2b} & 80801 & 7051 & 46480 & 1.687 & 1.003 & 0.015652 & 0.003182\\
& \textbf{4} & 98 & 450 & 603 & 1.106 & 0.065 & 0.008300 & 0.000945\\
\bottomrule
\end{tabular}
\end{small}
\bigskip

\caption{Timings for Magma's \textsc{EulerFactor} function, and implementations of Algorithm 7 in \textsc{Magma} and \textsc{C}.  Columns $\mu$ and $\sigma$ are means and standard deviations in milliseconds, running on a single core of an AMD EPYC 9754 2.25Ghz processor.}\label{tab:timings}
\end{table}

\begin{table}[tp]
\renewcommand{\arraystretch}{0.9}
\setlength{\tabcolsep}{3pt}
\begin{scriptsize}
\begin{tabular}{cclr}
type & $p$ & $C$&\\\toprule
\textbf{1} & 2095451 & \multicolumn{2}{l}{$y^2 = 65366932x^6+46833852x^5+950560081x^4+1014328354x^3-714563571x^2-632448x+750321408$}\\
&&&1.3 ms, 12.7$\mu$s\\\midrule
\textbf{1} & 2129069 & \multicolumn{2}{l}{$y^2 = -282619x^6+11424694x^5-66742653x^4+267785664x^3+783733439x^2-4055750250x-6492528143$}\\
&&&1.3 ms, 14.5$\mu$s\\\midrule
\textbf{1} & 2141299 & \multicolumn{2}{l}{$y^2 = 35664905x^6-1683266x^5+81620201x^4-43104564x^3-550491952x^2+869809612x-867569192$}\\
&&&1.4 ms, 12.3$\mu$s\\\midrule
\textbf{1} & 2192653 & \multicolumn{2}{l}{$y^2 = -7200195x^6+140298398x^5+82315425x^4+1863655712x^3+540423460x^2-272234940x-11070656$}\\
&&&1.4 ms, 12.4$\mu$s\\\midrule
\textbf{1} & 2192653 & \multicolumn{2}{l}{$y^2 = -86525452x^6+301639692x^5+165992173x^4-631039776x^3-602207136x^2+199197628x+586943224$}\\
&&&1.4 ms, 13.0$\mu$s\\\midrule
\textbf{2b} & 2192653 & \multicolumn{2}{l}{$y^2 = -7200195x^6+140298398x^5+82315425x^4+1863655712x^3+540423460x^2-272234940x-11070656$}\\
&&&1.4 ms, 12.4$\mu$s\\\midrule
\textbf{2b} & 2192653 & \multicolumn{2}{l}{$y^2 = -86525452x^6+301639692x^5+165992173x^4-631039776x^3-602207136x^2+199197628x+586943224$}\\
&&&1.4 ms, 13.0$\mu$s\\\midrule
\textbf{2b} & 3356999 & \multicolumn{2}{l}{$y^2 = -69840280x^6-88002004x^5+527168100x^4+947722520x^3-3244499x^2-2027697150x-67561855$}\\
&&&1.4 ms, 12.6$\mu$s\\\midrule
\textbf{2b} & 3365389 & \multicolumn{2}{l}{$y^2 = 3365504x^6+423061824x^5-732552719x^4+490957150x^3-2239426319x^2-172708548x-1287844864$}\\
&&&1.3 ms, 13.3$\mu$s\\\midrule
\textbf{2b} & 3520511 & \multicolumn{2}{l}{$y^2 = 87716080x^6-257007920x^5-436267519x^4+272523200x^3-1361866162x^2+320039284x-994564803$}\\
&&&1.4 ms, 12.6$\mu$s\\
\bottomrule
\end{tabular}
\end{scriptsize}
\medskip
\caption{
Examples that Magma's \textsc{EulerFactor} intrinsic took more than an hour to compute, with running times for Magma and~C implementations of Algorithm~\ref{alg:main}.}\label{tab:hard}
\end{table}

\begin{table}[bp]
\renewcommand{\arraystretch}{0.9}
\setlength{\tabcolsep}{3pt}
\begin{scriptsize}
\begin{tabular}{cclr}
type & $p$ & $C$&\\\toprule
\textbf{1} & 2095451 & \multicolumn{2}{l}{$y^2 = 65366932x^6+46833852x^5+950560081x^4+1014328354x^3-714563571x^2-632448x+750321408$}\\
&&&1.3 ms, 12.7$\mu$s\\\midrule
\textbf{1} & 2129069 & \multicolumn{2}{l}{$y^2 = -282619x^6+11424694x^5-66742653x^4+267785664x^3+783733439x^2-4055750250x-6492528143$}\\
&&&1.3 ms, 14.5$\mu$s\\\midrule
\textbf{1} & 2141299 & \multicolumn{2}{l}{$y^2 = 35664905x^6-1683266x^5+81620201x^4-43104564x^3-550491952x^2+869809612x-867569192$}\\
&&&1.4 ms, 12.3$\mu$s\\\midrule
\textbf{1} & 2192653 & \multicolumn{2}{l}{$y^2 = -7200195x^6+140298398x^5+82315425x^4+1863655712x^3+540423460x^2-272234940x-11070656$}\\
&&&1.4 ms, 12.4$\mu$s\\\midrule
\textbf{1} & 2192653 & \multicolumn{2}{l}{$y^2 = -86525452x^6+301639692x^5+165992173x^4-631039776x^3-602207136x^2+199197628x+586943224$}\\
&&&1.4 ms, 13.0$\mu$s\\\midrule
\textbf{2b} & 2239 & \multicolumn{2}{l}{$y^2 = 2720385x^6+58061748x^5+239277452x^4-714666410x^3+196933484x^2-986351148x+596368845$}\\
&&&6.3 ms, 16.3$\mu$s\\\midrule
\textbf{2b} & 2683 & \multicolumn{2}{l}{$y^2 = -6613595x^6-33446278x^5+32788943x^4+45761248x^3+96912643x^2-3579181026x+9931057425$}\\
&&&5.9 ms, 21.2$\mu$s\\\midrule
\textbf{2b} & 2833 & \multicolumn{2}{l}{$y^2 = 7977728x^6+73397364x^5+225744772x^4+359439708x^3-1032070399x^2+26023938x+656386269$}\\
&&&5.7 ms, 21.5$\mu$s\\\midrule
\textbf{2b} & 2957 & \multicolumn{2}{l}{$y^2 = -2208879x^6-9568852x^5+45531886x^4+1710470736x^3+1769873909x^2-4037475972x-1096633020$}\\
&&&5.7 ms, 22.2$\mu$s\\\midrule
\textbf{2b} & 3079 & \multicolumn{2}{l}{$y^2 = 28114349x^6+47422758x^5-91418589x^4-577694296x^3+735994923x^2+990569722x-1007267139$}\\
&&&6.1 ms, 20.0$\mu$s\\\bottomrule
\end{tabular}
\end{scriptsize}
\medskip
\caption{
Examples that Magma's \textsc{EulerFactor} intrinsic was unable to compute in 8 hours, with running times for our Magma and~C implementations of Algorithm~\ref{alg:main}.}\label{tab:fail}
\end{table}

\section{Declarations}

The authors declare that they have no competing interests that are relevant to the content of this article. A subset of the data used to test our implementation (including all the most difficult cases) is available in the \href{https://github.com/AndrewVSutherland/Genus2Euler}{\textsc{Genus2Euler}} GitHub repository.  The full dataset is available upon request.

\FloatBarrier

%%%%%%%%%%%%%%%%%%%%%%%%%%%%%%%%%
%%%%%%%%%%%%%%%%%%%%%%%%%%%%%%%%%

\end{document}